\documentclass[reqno]{amsart}
\usepackage{amscd}
\usepackage{amssymb}
\usepackage[margin=1in,includeheadfoot]{geometry}
\usepackage{cite}
\usepackage [latin1]{inputenc}

\allowdisplaybreaks[4]

\usepackage{tabularx,booktabs,tikz}
\usepackage{caption}
\usepackage{amsmath}
\usepackage{amsfonts}

\usepackage{amscd}
\usepackage{amsthm}
\usepackage{amssymb} \usepackage{latexsym}
\usepackage{eufrak}
\usepackage{euscript}
\usepackage{epsfig}
\usepackage{graphics}
\usepackage{array}
\usepackage{enumerate}
\usepackage{dsfont}
\usepackage{color}
\usepackage{wasysym}
\usepackage{hyperref}
\usepackage{pdfsync}

\def\beq{\begin{equation}}
\def\eeq{\end{equation}}
\def\ba{\begin{array}}
\def\ea{\end{array}}

\newtheorem{thm}{Theorem}[section]
\newtheorem{lm}[thm]{Lemma}
\newtheorem{prop}[thm]{Proposition}
\newtheorem{crl}[thm]{Corollary}

\theoremstyle{definition}

\newtheorem{df}[thm]{Definition}

\theoremstyle{remark}

\begin{document}
\pagestyle{plain}
\title{Existence and multiplicity of solutions to discrete fractional logarithmic Kirchhoff equations}

\author{Lidan Wang}
\email{wanglidan@ujs.edu.cn}
\address{Lidan Wang: School of Mathematical Sciences, Jiangsu University, Zhenjiang 212013, People's Republic of China}

\begin{abstract}
In this paper, we study the discrete fractional logarithmic Kirchhoff equation
$$
\left(a+b \int_{\mathbb{Z}^d}|\nabla^s u|^{2} d \mu\right) (-\Delta)^s u+h(x) u=|u|^{p-2}u \log u^{2}, \quad x\in \mathbb{Z}^d,
$$
where $a,\,b>0$ and $0<s<1$. Under suitable assumptions on $h(x)$, we first prove the existence of ground state solutions by the mountain-pass theorem for $p>4$; then we verify the existence of ground state sign-changing solutions based on the method of Nehari manifold for $p>6$. Finally, we establish the multiplicity of nontrivial weak solutions.
\end{abstract}

\maketitle


\
\


\section{Introduction}

The fractional Laplacian, understood as a positive power of the classical Laplacian, has a wide range of applications arising in some physical
phenomena such as fractional quantum mechanics, flames propagation, see \cite{L8,DP}.
In the last decades, a lot of attention has been focused on the problems involving fractional Laplace operators and Kirchhoff-type nonlocal terms,
$$\left(a+b \int_{\mathbb{R}^d \times \mathbb{R}^d}\frac{|u(x)-u(y)|^2}{|x-y|^{d+2s}}\,dxdy\right) (-\Delta)^s u+h(x) u=g(u),$$
where $a,\,b>0$ and $s\in (0,1)$. This fractional Kirchhoff equation was first introduced in \cite{FV}. After that, many remarkable results have been yielded, see \cite{A1,GY,I1,SC,WG,W9} and the references therein. 

As we know, the logarithmic nonlinearity $g(u)=|u|^{p-2}u \log u^{2}$ has many applications in quantum optics, quantum
mechanics, transport, nuclear physics and diffusion phenomena etc, see \cite{Z1}. This makes many scholars study the Kirchhoff-type problems with logarithmic nonlinearity. For the logarithmic Schr\"{o}dinger equations, we refer the readers to \cite{AJ1,AJ,FT,S,ZZ2}. For the logarithmic Kirchhoff equations, we refer the readers to \cite{GJ,HG,LWZ,OA,WTC,WL}. For the fractional logarithmic Kirchhoff equations, we refer the readers to \cite{FW,HS,L9,YB,XH}. 

In recent years, there are many works on graphs, see for examples \cite{HSZ,HLW,HW,LW,SYZ,WW,YZ}.  For the discrete logarithmic Sch\"{o}dinger equations on graphs, we refer the readers to \cite{CWY,CR,HJ,OZ,SY}. For the discrete Kirchhoff equations on graphs, we refer the readers to \cite{CY,L1,PJ,W5,W3}. Very recently, Wang \cite{W2} studied the discrete logarithmic Kirchhoff equations and proved the existence and asymptotic behavior of least energy sign-changing solutions.

Recently, Zhang, Lin and Yang \cite{ZLY} established a discrete version of the fractional Laplace operator $(-\Delta)^s$ through
the heat semigroup on a stochastically complete, connected, locally finite graph. Based on this definition, they obtained the multiplicity solutions to a discrete fractional Schr\"{o}dinger equation. Wang \cite{W7} established the existence and multiplicity of solutions to a discrete fractional Schr\"{o}dinger equation on lattice graphs. Very recently, Wang \cite{W6} considered a fractional logarithmic Schr\"{o}dinger equation and proved the existence of ground state solutions and ground state sign-changing solutions.  However, to the best of our knowledge, there seems no results for discrete fractional 
logarithmic Kirchhoff-type problems on graphs. Motivated by the aforementioned works, in this paper, we will consider the fractional logarithmic Kirchhoff equations on lattice graphs and study the existence of ground state solutions and ground state sign-changing solutions. More precisely, we consider the following fractional logarithmic Kirchhoff equation
\begin{equation}\label{0.2}
\left(a+b \int_{\mathbb{Z}^d}|\nabla^s u|^{2} d \mu\right) (-\Delta)^s u+h(x)u=|u|^{p-2}u \log u^{2}, \quad x\in\mathbb{Z}^d,
\end{equation}
where $a,\,b>0$ and $s\in (0,1)$. Here the fractional Laplace operator is defined as
$$
(-\Delta)^{s} u(x)=\sum_{y \in \mathbb{Z}^d, y \neq x} w_{s}(x, y)(u(x)-u(y)),
$$ 
and 
$$\left|\nabla^{s} u(x)\right|^2=\frac{1}{2} \sum\limits_{y \in \mathbb{Z}^d, y \neq x}w_{s}(x, y)(u(x)-u(y))^{2},$$
where $w_s(x,y)$ is a symmetric positive function satisfying
$$
c_{s,d}|x-y|^{-d-2s}\leq w_s(x,y)\leq C_{s,d}|x-y|^{-d-2s},\quad x\neq y.
$$
We always assume that the potential $h$ satisfies 
\begin{itemize}
    \item[$(h_1)$] for any $x\in\mathbb{Z}^d$, there exists a constant $h_0>0$ such that $h(x) \geq h_0$;
\item[($h_2$)] there exists a point $x_0\in\mathbb{Z}^d$ such that $h(x)\rightarrow\infty$ as $|x-x_0|\rightarrow\infty.$
\end{itemize}

For any $q>p$, it is obvious that 
$$\lim\limits_{t\rightarrow 0}\frac{t^{p-1}\log t^2}{t}=0, \qquad\text{and}\qquad \lim\limits_{t\rightarrow \infty}\frac{t^{p-1}\log t^2}{t^{q-1}}=0,$$
which implies that for any $\varepsilon>0$ , there exists $C_\varepsilon>0$ such that
\begin{equation}\label{1.2}
|t|^{p-1}|\log t^2|\leq \varepsilon |t|+C_\varepsilon |t|^{q-1}, \quad t\neq 0.
\end{equation}

Now we state our main results.

\begin{thm}\label{t1}
Let $p>4$ and $(h_1)$-$(h_2)$ hold. Then the equation (\ref{0.2}) has a ground state solution $u$.
\end{thm}

\begin{thm}\label{t2}
  Let $p>6$ and $(h_1)$-$(h_2)$ hold. Then the equation (\ref{0.2}) has a ground state sign-changing solution $v$.   
\end{thm}

\begin{thm}\label{t3}
Let $p>6$ and $(h_1)$-$(h_2)$ hold. For the ground state solution $u$ and the ground state sign-changing solution $v$, we have $J_{s,2}(v)>2J_{s,2}(u)$, where $J_{s,2}$ is the functional related to the equation (\ref{0.2}). As a consequence, the equation (\ref{0.2}) has at least four different nontrivial weak solutions. 
\end{thm}

\
\

This paper is organized as follows. In Section 2, we state some basic results on graphs. In Section 3, we prove the existence of ground state solutions (Theorem \ref{t1}). In Section 4, we first prove the existence of ground state sign-changing solutions (Theorem \ref{t2}). Then we prove the multiplicity of nontrivial weak solutions (Theorem \ref{t3}).  

\section{Preliminaries}
In this section, we introduce the basic settings on graphs and give some preliminary results. 

Let $G=(V, E)$ be a connected, locally finite graph, where $V$ denotes the vertex set and $E$ denotes the edge set. We call vertices $x$ and $y$ neighbors, denoted by $x \sim y$, if there exists an edge connecting them, i.e. $(x, y) \in E$. For any $x,y\in V$, the distance $d(x,y)$ is defined as the minimum number of edges connecting $x$ and $y$, namely
$$d(x,y)=\inf\{k:x=x_0\sim\cdots\sim x_k=y\}.$$

In this paper, we consider, the natural discrete model of the Euclidean space, the integer lattice graph.  The $d$-dimensional integer lattice graph, denoted by $\mathbb{Z}^d$, consists of the set of vertices $V=\mathbb{Z}^d$ and the set of edges $E=\{(x,y): x,\,y\in~V,\,\underset {{i=1}}{\overset{d}{\sum}}|x_{i}-y_{i}|=1\}.$
We will always denote $|x-y|:=d(x,y)$ on the lattice graph $V$.

Let $C(\mathbb{Z}^d)$ be the set of all functions on $\mathbb{Z}^d$. For $u,v \in C(\mathbb{Z}^d)$ and $s\in (0,1)$, as in \cite{ZLY}, we define the fractional gradient form as
\begin{equation*}
\nabla^{s} u \nabla^{s} v(x)=\frac{1}{2} \sum_{y \in \mathbb{Z}^d, y \neq x} w_{s}(x, y)\left(u(x)-u(y)\right)\left(v(x)-v(y)\right),
\end{equation*}
where $w_s(x,y)$ is a symmetric positive function satisfying
\begin{equation}\label{25}
c_{s,d}|x-y|^{-d-2s}\leq w_s(x,y)\leq C_{s,d}|x-y|^{-d-2s},\quad x\neq y.    
\end{equation}
This order estimate can be seen in \cite[Theorem 2.4]{W0}. Moreover, we write the length of $\nabla^{s} u(x)$ as
$$
\left|\nabla^{s} u\right|(x)=\sqrt{\nabla^{s} u \nabla^{s} u(x)}=\left(\frac{1}{2} \sum_{y \in \mathbb{Z}^d, y \neq x} w_{s}(x, y)(u(x)-u(y))^{2}\right)^{\frac{1}{2}} .
$$
Let the fractional Laplacian of $u\in  C(\mathbb{Z}^d)$ be defined by
$$
(-\Delta)^{s} u(x)= \sum_{y \in \mathbb{Z}^d, y \neq x} w_{s}(x, y)(u(x)-u(y)).
$$ 

The space $\ell^{p}(\mathbb{Z}^d)$ is defined as $
\ell^{p}(\mathbb{Z}^d)=\left\{u \in C(\mathbb{Z}^d):\|u\|_{p}<\infty\right\},
$ where
$$
\|u\|_{p}= \begin{cases}\left(\sum\limits_{x \in \mathbb{Z}^d}|u(x)|^{p}\right)^{\frac{1}{p}}, &  1 \leq p<\infty, \\ \sup\limits_{x \in \mathbb{Z}^d}|u(x)|, & p=\infty.\end{cases}
$$
For convenience, for $u\in C(\mathbb{Z}^{d})$, we always write $
\int_{\mathbb{Z}^{d}} u(x)\,d \mu=\sum\limits_{x \in \mathbb{Z}^{d}} u(x),
$
where $\mu$ is the counting measure on $\mathbb{Z}^{d}$.

Let $C_{c}(\mathbb{Z}^d)$ be the set of all functions on $\mathbb{Z}^d$ with finite support. Now we state a fractional Sobolev space. For $s\in(0,1)$, let $W^{s,2}(\mathbb{Z}^d)$ be the completion of $C_c(\mathbb{Z}^d)$ under
the norm
$$
\|u\|_{W^{s,2}}=\left(\int_{\mathbb{Z}^d}\left(|\nabla^s u|^{2}+u^{2}\right) d \mu\right)^{\frac{1}{2}}.
$$
It is clear that $W^{s,2}(\mathbb{Z}^d)$ is a Hilbert space with the inner product
$$
\langle u, v\rangle_{W^{s,2}}=\int_{\mathbb{Z}^d}\left(\nabla^su \nabla^sv+u v\right)\,d \mu.
$$
 In order to establish our results, we introduce a subspace of $W^{s,2}(\mathbb{Z}^d)$, namely
$$
H^{s,2}(\mathbb{Z}^d)=\left\{u \in W^{s,2}(\mathbb{Z}^d): \int_{\mathbb{Z}^d} h(x) u^{2}\,d \mu<\infty\right\}
$$
equipped with the norm
$$
\|u\|_{H^{s,2}}=\left(\int_{\mathbb{Z}^d}\left(a|\nabla^s u|^{2}+h(x)u^{2}\right)\,d \mu\right)^{\frac{1}{2}}.
$$
Since $C_c(\mathbb{Z}^d)\subset H^{s,2}(\mathbb{Z}^d)$, the space $H^{s,2}(\mathbb{Z}^d)\neq\emptyset.$ Moreover, we have that $H^{s,2}(\mathbb{Z}^d)$ ($H^{s,2}$ for brevity) is also a Hilbert space with the inner product
$$
\langle u, v\rangle_{H^{s,2}}=\int_{\mathbb{Z}^d}\left(a\nabla^s u \nabla^s v+h(x) u v\right)\,d \mu.
$$ By $(h_1)$, we have that $$\int_{\mathbb{Z}^d}u^2\,d\mu\leq \frac{1}{h_0}\int_{\mathbb{Z}^d} h(x)u^2\,d\mu,$$
and hence, for $q\geq 2$,
\begin{equation}\label{ac}
    \|u\|_q\leq \|u\|_2\leq \|u\|_{W^{s,2}}\leq C_{a,h_0}\|u\|_{H^{s,2}}.
\end{equation}

\ 
\

The energy functional $J_{s,2}: H^{s,2}\rightarrow \mathbb{R}$ related to the equation (\ref{0.2}) is given by
\begin{equation*}
J_{s,2}(u)=\frac{1}{2} \int_{\mathbb{Z}^d}\left(a|\nabla^s u|^{2}+h(x) u^{2}\right) d \mu+\frac{b}{4}\left(\int_{\mathbb{Z}^d}|\nabla^s u|^{2} d \mu\right)^{2}+\frac{2}{p^2}\int_{\mathbb{Z}^d}|u|^p\,d\mu-\frac{1}{p} \int_{\mathbb{Z}^d} |u|^{p} \log u^{2} d \mu. 
\end{equation*}
By (\ref{1.2}), we derive that $J_{s,2}\in C^1(H^{s,2},\mathbb{R})$ and, for any $\phi\in H^{s,2}$,
\begin{eqnarray*}
\langle J'_{s,2}(u),\phi\rangle=\int_{\mathbb{Z}^d}\left(a \nabla^s u \nabla^s \phi+h(x) u \phi\right)\,d \mu+b \int_{\mathbb{Z}^d}|\nabla^s u|^{2}\,d \mu \int_{\mathbb{Z}^d} \nabla^s u \nabla^s \phi\,d \mu-\int_{\mathbb{Z}^d}|u|^{p-2}u\phi\log u^2\,d\mu.
\end{eqnarray*}

\begin{df}
We say $u \in H^{s,2}$ is a nontrivial weak solution to the equation (\ref{0.2}) if $u\neq 0$ is a critical point of the functional $J_{s,2}$. 
We say that $u\in H^{s,2}$ is a ground state solution to the equation (\ref{0.2}) if
$u$ is a nontrivial solution satisfying
$$
J_{s,2}(u)=\inf\limits_{v\in \mathcal{N}_{s,2}} J_{s,2}(v),
$$
where $\mathcal{N}_{s,2}=\{v\in H^{s,2}\backslash\{0\}: \langle J_{s,2}'(v),v\rangle=0\}$ is the Nehari manifold. 

We say that $u\in H^{s,2}$ is a sign-changing solution to the equation (\ref{0.2}) if $u$ is a weak solution to the equation (\ref{0.2}) and $u^{\pm}\not\equiv 0$, where $u^{+}=\max \{u, 0\}$ and $u^{-}=\min \{u, 0\}$. We say that $u\in H^{s,2}$ is a ground state sign-changing solution to the equation (\ref{0.2}) if $u$ is a sign-changing solution satisfying
$$
J_{s,2}(u)=\inf\limits_{v\in \mathcal{M}_{s,2}} J_{s,2}(v),
$$
where $\mathcal{M}_{s,2}=\left\{v \in H^{s,2}: v^{ \pm} \neq 0 \text { and } \langle J_{s,2}^{\prime}(v), v^{+}\rangle=\langle J_{s,2}^{\prime}(v),v^{-}\rangle=0\right\}$ is the sign-changing Nehari set.
\end{df}

Now we give some basic lemmas in this paper. The first one is about the formulas of integration by parts, see \cite{ZLY}.

\begin{lm}\label{l0} 
Let $u \in W^{s,2}(\mathbb{Z}^d)$. Then for any $\phi \in C_{c}(\mathbb{Z}^d)$, we have
$$
\int_{\mathbb{Z}^d} \phi(-\Delta)^{s} u\,d \mu=\int_{\mathbb{Z}^d}\nabla^{s} u \nabla^{s} \phi\,d \mu.
$$
\end{lm}

\begin{lm}\label{li}
  If $u \in H^{s,2}$ is a weak solution to the equation (\ref{0.2}), then $u$ is a pointwise solution to the equation (\ref{0.2}).  
\end{lm}

\begin{proof}
If $u \in H^{s,2}$ is a weak solution to the equation (\ref{0.2}), then for any $\phi \in H^{s,2}$, there holds
$$
\int_{\mathbb{Z}^d}(a \nabla^s u \nabla^s \phi+h(x) u \phi)\,d \mu+b \int_{\mathbb{Z}^d}|\nabla^s u|^{2}\,d \mu \int_{\mathbb{Z}^d} \nabla^s u \nabla^s \phi\,d \mu=\int_{\mathbb{Z}^d}|u|^{p-2}u\phi\log u^2\,d\mu.
$$
Since $C_{c}(\mathbb{Z}^d)$ is dense in $H^{s,2}$, for any $\phi \in C_{c}(\mathbb{Z}^d)$, by integration by parts, we have
\begin{equation}\label{2.1}
\left(a+b \int_{\mathbb{Z}^d}|\nabla^s u|^{2} d \mu\right) \int_{\mathbb{Z}^d}(-\Delta)^s u\,\phi\,d\mu+\int_{\mathbb{Z}^d}h(x) u \phi\,d \mu=\int_{\mathbb{Z}^d}|u|^{p-2}u\phi \log u^{2} d \mu.
\end{equation}
For any fixed $x_0\in \mathbb{Z}^d$, let
$$
\phi_0(x)= \begin{cases}1, & x=x_0 \\ 0, & x \neq x_0.\end{cases}
$$
Clearly, $\phi_0 \in C_c(\mathbb{Z}^d)$. By taking $\phi_0$ as a test function in (\ref{2.1}), we derive that
$$\left(a+b \int_{\mathbb{Z}^d}|\nabla^s u(x)|^{2} d \mu\right)(-\Delta)^s u(x_0)+ h (x_0)        u(x_0)=|u(x_0)|^{p-2}u(x_0)\log u^{2}(x_0).$$  We get that $u$ is a pointwise solution to the equation (\ref{0.2}) by the arbitrariness of $x_0$.

\end{proof}

\begin{lm}\label{i1}
Let $u\in H^{s,2}$. Then we have $u^-,u^+, |u|\in H^{s,2}$.   
\end{lm}

\begin{proof}
Note that $u=u^{+}+u^{-}$ and  $u^{+} u^{-}=0$. A direct calculation yields that
\begin{align*}\label{ar}
\nabla^{s} u^{+} \nabla^{s} u^{-}(x) & =\frac{1}{2} \sum_{y\in\mathbb{Z}^d, y \neq x} w_{s}(x, y)\left(u^{+}(x)-u^{+}(y)\right)\left(u^{-}(x)-u^{-}(y)\right) \nonumber\\
& =-\frac{1}{2} \sum_{y\in\mathbb{Z}^d, y \neq x} w_{s}(x, y)\left(u^{+}(x) u^{-}(y)+u^{+}(y) u^{-}(x)\right) \\
& \geq 0 \nonumber.
\end{align*}
Therefore, we get that
\begin{align*}
 |\nabla^s u|^{2}&=\frac{1}{2} \sum_{y\in\mathbb{Z}^d, y \neq x} w_{s}(x, y)\left(u(x)-u(y)\right)^2 \\
&=\frac{1}{2} \sum_{y\in\mathbb{Z}^d, y \neq x} w_{s}(x, y)\left(u^+(x)-u^+(y)+u^-(x)-u^-(y)\right)^2 \\
&=\frac{1}{2} \sum_{y\in\mathbb{Z}^d, y \neq x} w_{s}(x, y)\left[(u^+(x)-u^+(y))^2+(u^-(x)-u^-(y))^2+2(u^+(x)-u^+(y))(u^-(x)-u^-(y))\right] \\
&= \left|\nabla^s u^{+}\right|^{2}+\left|\nabla^s u^{-}\right|^{2}+2 \nabla^s u^{+} \nabla ^su^{-} \\ &\geq\left|\nabla^s u^{-}\right|^{2}.   
\end{align*}
Moreover, we have $\left|u^{-}\right|^{2} \leq|u|^{2}$. Then for $u \in H^{s,2}$,  we obtain that $u^{-} \in H^{s,2}$.
A similar argument tells us that $u^+\in H^{s,2}$. Since $H^{s,2}$ is a Hilbert space and $|u|=u^+-u^-$, we get that $|u|\in H^{s,2}.$
\end{proof}

The following one is a compactness result related to the space $H^{s,2}$.

\begin{lm}\label{lg}
Let $(h_1)$-$(h_2)$ hold. Then the space $H^{s,2}$ is embedded compactly into $\ell^{q}(\mathbb{Z}^d)$ for $q\geq 2$. That is, for any bounded sequence $\left\{u_{k}\right\} \subset H^{s,2}$, there exists $u \in H^{s,2}$ such that, up to a subsequence, 
 $$
\begin{cases}u_k \rightharpoonup u, & \text { weakly in } H^{s,2}, \\ u_k \rightarrow u, & \text { pointwise in } \mathbb{Z}^d, \\ u_k \rightarrow u, & \text { strongly in } \ell^{q}(\mathbb{Z}^d),~q \in[2,\infty].\end{cases}
$$
\end{lm}

\begin{proof}
    The proof is similar to that of \cite{ZLY}. We omit it here.
\end{proof}

\begin{prop}\label{o} Let $r,t>0$. Then for any $u\in H^{s,2}$, we have
  \begin{itemize}
      \item[(i)] $$
      \int_{\mathbb{Z}^d} \left|\nabla^s(ru^{+}+tu^{-})\right|^2 d \mu=r^2\int_{\mathbb{Z}^d} \left|\nabla^s u^{+}\right|^2 d \mu+t^2\int_{\mathbb{Z}^d}\left|\nabla^s u^{-}\right|^2 d \mu-rtK(u),
      $$
 where $K(u)=\sum\limits_{x \in \mathbb{Z}^d} \sum\limits_{y\in \mathbb{Z}^d,y\neq x}w_{x}(x,y)\left[u^{+}(y) u^{-}(x)+u^{-}(y) u^{+}(x)\right]\leq 0;$    
 \item[(ii)] $$
 \int_{\mathbb{Z}^d} \nabla^s\left(ru^{+}+tu^{-}\right)\nabla^s (ru^+)\, d \mu=r^2\int_{\mathbb{Z}^d} \left|\nabla^s u^{+}\right|^2 d \mu-\frac{rt}{2} K(u);  
 $$
 \item[(iii)] $$
 \int_{\mathbb{Z}^d} \nabla^s\left(ru^{+}+tu^{-}\right)\nabla^s (tu^-)\, d \mu=t^2\int_{\mathbb{Z}^d} \left|\nabla^s u^{-}\right|^2 d \mu-\frac{rt}{2} K(u).
$$
  \end{itemize}  
\end{prop}

\begin{proof}
(i) A straightforward  calculation gives that
$$
\begin{aligned}
& \int_{\mathbb{Z}^d} \left|\nabla^s(ru^{+}+tu^{-})\right|^2 d \mu \\
= & \frac{1}{2} \sum_{x \in \mathbb{Z}^d} \sum_{y\in\mathbb{Z}^d, y\neq x}w_{s}(x,y)\left[\left(ru^{+}+tu^{-}\right)(x)-\left(ru^{+}+tu^{-}\right)(y)\right]^{2}\\
= & \frac{1}{2} \sum_{x \in \mathbb{Z}^d} \sum_{y\in\mathbb{Z}^d, y\neq x}w_s(x,y)\left[\left(ru^{+}(x)-ru^{+}(y)\right)^{2}+\left(tu^{-}(x)-tu^{-}(y)\right)^{2}-2rt\left[u^{+}(y) u^{-}(x)+u^{-}(y) u^{+}(x)\right]\right] \\
= & r^2\int_{\mathbb{Z}^d} \left|\nabla^s u^{+}\right|^2 d \mu+t^2\int_{\mathbb{Z}^d}\left|\nabla^s u^{-}\right|^2 d \mu-rtK(u).
\end{aligned}
$$

(ii) By a direct computation, we get that
$$
\begin{aligned}
&\int_{\mathbb{Z}^d} \nabla^s\left(ru^{+}+tu^{-}\right)\nabla^s (ru^+)\, d \mu \\= & \frac{1}{2} \sum_{x \in \mathbb{Z}^d} \sum_{y\in\mathbb{Z}^d, y\neq x}w_s(x,y)\left[\left(ru^{+}+tu^{-}\right)(x)-\left(ru^{+}+tu^{-}\right)(y)\right]\left[ru^{+}(x)-ru^{+}(y)\right] \\
= & \frac{1}{2} \sum_{x \in \mathbb{Z}^d} \sum_{y\in\mathbb{Z}^d, y\neq x}w_s(x,y)\left[\left(ru^{+}(x)-ru^+(y)\right)^{2}-rt\left[u^{+}(y) u^{-}(x)+u^{-}(y) u^{+}(x)\right]\right] \\
= & r^2\int_{\mathbb{Z}^d} \left|\nabla^su^{+}\right|^2 d \mu-\frac{rt}{2} K(u) .
\end{aligned}
$$

(iii) The proof is similar to that of (ii), we omit it here.

\end{proof}

\section{Ground state solutions}
In this section, we prove Theorem \ref{t1} by the mountain-pass theorem. Recall that, for a given functional $\Phi\in C^{1}(X,\mathbb{R})$, a sequence $\{u_k\}\subset X$ is a Palais-Smale sequence at level $c\in\mathbb{R}$, $(PS)_c$ sequence for short, of the functional $\Phi$, if it satisfies, as $k\rightarrow\infty$,
\begin{eqnarray*}
\Phi(u_k)\rightarrow c, \qquad \text{in}~ X,\qquad\text{and}\qquad
\Phi'(u_k)\rightarrow 0, \qquad \text{in}~X^*
\end{eqnarray*}
where $X$ is a Banach space and $X^{*}$ is the dual space of $X$. Moreover, we say that $\Phi$ satisfies $(PS)_c$ condition, if any $(PS)_c$ sequence has a convergent subsequence. 

First, we show that the functional $J_{s,2}$ satisfies the mountain-pass geometric structure.
\begin{lm}\label{lm}
  Let $p>4$ and $(h_1)$-$(h_2)$ hold. Then 
  \begin{itemize}
      \item [(i)] there exist $\sigma, \rho>0$ such that $J_{s,2}(u) \geq \sigma$ for $\|u\|_{H^{s,2}}=\rho$;
      \item[(ii)] there exists $e\in H^{s,2}$ with $\|e\|_{H^{s,2}}>\rho$ such that $J_{s,2}(e) <0$.
  \end{itemize}
\end{lm}
 
\begin{proof}
(i) For $u\in H^{s,2}$, by (\ref{1.2}) and (\ref{ac}), we have that
\begin{equation}\label{98}
\begin{aligned}
\frac{1}{p}\int_{\mathbb{Z}^d}|u|^{p} \log |u|^{2} d \mu  \leq &\int_{\mathbb{Z}^d}\left(\varepsilon|u|^{2} +C_\varepsilon|u|^q\right)d \mu\\ \leq &  \varepsilon\|u\|_2^2 +C_\varepsilon \|u\|^q_q\\ \leq &\varepsilon \|u\|^2_{H^{s,2}}+C_\varepsilon \|u\|^q_{H_{s,2}}.
\end{aligned} 
\end{equation}
Let $\varepsilon=\frac{1}{4},$ then we obtain that
\begin{align*}
 J_{s,2}(u) 
& =\frac{1}{2}\|u\|_{H^{s,2}}^{2}+\frac{b}{4}\left(\int_{\mathbb{Z}^d}|\nabla^s u|^2\,d\mu\right)^2+\frac{2}{p^2}\int_{\mathbb{Z}^d}|u|^p\,d\mu-\frac{1}{p} \int_{\mathbb{Z}^d}|u|^{p} \log |u|^{2} d \mu \\
& \geq \frac{1}{2}\|u\|_{H^{s,2}}^{2}-\frac{1}{4}\|u\|_{H^{s,2}}^{2}-C \|u\|_{H^{s,2}}^{q} \\
& \geq \frac{1}{4}\|u\|_{H^{s,2}}^{2}-C \|u\|_{H^{s,2}}^{q}.   
\end{align*}
Since $q>p>4$, there exist $\sigma, \rho>0$ small enough such that $J_{s,2}(u)\geq \sigma>0$ for $\|u\|_{H^{s,2}}=\rho.$

(ii) Let $u\in H^{s,2}\backslash \{0\}$ be fixed.  For $t>0$, since $p>4$, we have that
\begin{equation}\label{23}
 \begin{aligned}
    J_{s,2}(tu) 
& =\frac{t^{2}}{2}\|u\|_{H^{s,2}}^{2}+\frac{bt^4}{4}\left(\int_{\mathbb{Z}^d}|\nabla^s u|^2\,d\mu\right)^2+\frac{2}{p^2}\int_{\mathbb{Z}^d}|tu|^p\,d\mu-\frac{1}{p}\int_{\mathbb{Z}^d}|t u|^{p} \log |t u|^{2} d \mu\\ & =\frac{t^{2}}{2}\|u\|_{H^{s,2}}^{2}+\frac{bt^4}{4}\left(\int_{\mathbb{Z}^d}|\nabla^s u|^2\,d\mu\right)^2-\frac{t^p}{p^2}\left(p\int_{\mathbb{Z}^d}|u|^{p} \log |t u|^{2} d \mu-2\int_{\mathbb{Z}^d}|u|^p\,d\mu\right)\\ &=\frac{t^{2}}{2}\|u\|_{H^{s,2}}^{2}+\frac{bt^4}{4}\left(\int_{\mathbb{Z}^d}|\nabla^s u|^2\,d\mu\right)^2-\frac{t^p}{p^2}\left[\left(p\log t^2-2\right)\int_{\mathbb{Z}^d}|u|^p\,d\mu+p\int_{\mathbb{Z}^d}|u|^{p} \log |u|^{2} d \mu\right]
\\ &\rightarrow- \infty, \quad t\rightarrow \infty.
\end{aligned}   
\end{equation}
Therefore, there exists $t_{0}>0$ large enough such that $\left\|t_{0} u\right\|>\rho$ and $J_{s,2}\left(t_{0} u\right)<0$. The proof is completed by taking $e=t_{0}u$. 

\end{proof}

In the following, we verify the compactness of Palais-Smale sequence. Namely $J_{s,2}$ satisfies the $(P S)$ condition.

\begin{lm}\label{lb}
Let $p>4$ and $(h_1)$-$(h_2)$ hold. Then for any $c \in \mathbb{R}$, $J_{s,2}$ satisfies the $(P S)_{c}$ condition.  
\end{lm}

\begin{proof}
 Let $\left\{u_{k}\right\} \subset H^{s,2}$ be a $(PS)_c$ sequence, namely $J_{s,2}\left(u_{k}\right) \rightarrow c$ and $J_{s,2}^{\prime}\left(u_{k}\right) \rightarrow 0$ as $k \rightarrow\infty$. Note that $p>4$, then we deduce that
$$
\begin{aligned}
c+1+\left\|u_{k}\right\|_{H^{s,2}} & \geq J_{s,2}\left(u_{k}\right)-\frac{1}{p}\left\langle J_{s,2}^{\prime}\left(u_{k}\right), u_{k}\right\rangle \\
& =\left(\frac{1}{2}-\frac{1}{p}\right)\left\|u_{k}\right\|_{H^{s,2}}^{2}+b\left(\frac{1}{4}-\frac{1}{p}\right)\left(\int_{\mathbb{Z}^d}|\nabla^s u_k|^2\,d\mu\right)^2+\frac{2}{p^{2}} \int_{\mathbb{Z}^d}\left|u_{k}\right|^{p} d \mu \\
& \geq\left(\frac{1}{2}-\frac{1}{p}\right)\left\|u_{k}\right\|_{H^{s,2}}^{2}.
\end{aligned}
$$
This inequality implies that $\left\{u_{k}\right\}$ is bounded in $H^{s,2}$. By Lemma \ref{lg}, up to a subsequence, there exists $u \in H^{s,2}$ such that 
 \begin{equation}\label{24}
\begin{cases}u_k \rightharpoonup u, & \text { weakly in } H^{s,2}, \\ u_k \rightarrow u, & \text { pointwise in } \mathbb{Z}^d, \\ u_k \rightarrow u, & \text { strongly in } \ell^{q}(\mathbb{Z}^d),~q \in[2,\infty].\end{cases}
 \end{equation}
By (\ref{25}), we derive that
 \begin{align} \label{26}
    \int_{\mathbb{Z}^d} |\nabla^s u|^2\,d\mu=\nonumber&\frac{1}{2}\sum\limits_{x\in\mathbb{Z}^d} \sum\limits_{y \in \mathbb{Z}^d, y \neq x}w_{s}(x, y)(u(y)-u(x))^{2}\\ \leq\nonumber& C_{s,d}\sum\limits_{x\in\mathbb{Z}^d} \sum\limits_{y \in \mathbb{Z}^d, y \neq x}\frac{|u(x)|^2+|u(y)|^2}{|x-y|^{d+2s}}\\=\nonumber&2C_{s,d}\sum\limits_{x\in\mathbb{Z}^d}|u(x)|^2\sum\limits_{y \in \mathbb{Z}^d, y \neq x}\frac{1}{|x-y|^{d+2s}}\\ =\nonumber&2C_{s,d}\sum\limits_{z\neq 0}\frac{1}{|z|^{d+2s}}\sum\limits_{x\in\mathbb{Z}^d}|u(x)|^2\\=\nonumber& C\int_{\mathbb{Z}^d}|u(x)|^2\,d\mu,
\end{align}     
 where $C=2C_{s,d}\sum\limits_{z\neq 0}\frac{1}{|z|^{d+2s}}<\infty$. As a consequence, it follows from the H\"{o}lder inequality, the boundedness of $\{u_k\}$ and (\ref{24}) that
 $$
 \begin{aligned}
\int_{\mathbb{Z}^d} |\nabla^s u_k||\nabla^s (u_k-u)|\,d\mu \leq & \left(\int_{\mathbb{Z}^d} |\nabla^s u_k|^2\,d\mu\right)^{\frac{1}{2}}\left(\int_{\mathbb{Z}^d} |\nabla^s (u_k-u)|^2\,d\mu\right)^{\frac{1}{2}}\\ \leq & C\|u_k\|_{H^{s,2}}\|u_k-u\|_{2}
\\ \rightarrow & 0, \quad k\rightarrow\infty. \end{aligned}
 $$
Moreover, by (\ref{1.2}), (\ref{ac}), the H\"{o}lder inequality, the boundedness of $\{u_k\}$ and (\ref{24}), we get that
$$
\begin{aligned}
&\left|\int_{\mathbb{Z}^d}|u_k|^{p-2}u_k(u_k-u)\log u_k^2\,d\mu\right|\\ \leq &\int_{\mathbb{Z}^d} |u_k|^{p-1}\left|\log u_k^2\right||u_k-u|\,d\mu\\ \leq &\int_{\mathbb{Z}^d}\left(\varepsilon|u_k|+C_\varepsilon|u_k|^{q-1}\right)|u_k-u|\,d\mu \\ \leq &\varepsilon\left(\int_{\mathbb{Z}^d}|u_k|^2\,d\mu\right)^{\frac{1}{2}}\left(\int_{\mathbb{Z}^d}|u_k-u|^2\,d\mu\right)^{\frac{1}{2}}+C_\varepsilon\left(\int_{\mathbb{Z}^d}|u_k|^q\,d\mu\right)^{\frac{q-1}{q}}\left(\int_{\mathbb{Z}^d}|u_k-u|^q\,d\mu\right)^{\frac{1}{q}}\\ \leq &\varepsilon\|u_k\|_{H^{s,2}}\|u_k-u\|_{2}+C_\varepsilon \|u_k\|^{q-1}_{H^{s,2}}\|u_k-u\|_{q}\\ \rightarrow & 0,\quad k\rightarrow\infty.
\end{aligned}
$$
The arguments above tell us that
\begin{align*}
  &\left|\langle u_k,u_k-u\rangle_{H^{s,2}}\right|\\ &\leq |\langle J_{s,2}'(u_k),u_k-u\rangle|+b\int_{\mathbb{Z}^d}|\nabla^s u_k|^{2}\,d \mu\int_{\mathbb{Z}^d}|\nabla^s u_k||\nabla^s (u_k-u)|\,d \mu+\left|\int_{\mathbb{Z}^d}|u_k|^{p-2}u_k(u_k-u)\log u_k^2\,d\mu\right|\nonumber\\&\leq o_k(1)\|u_k-u\|_{H^{s,2}}+b\|u_k\|_{H^{s,2}}^2\int_{\mathbb{Z}^3}|\nabla^s u_k||\nabla^s (u_k-u)|\,d \mu+\left|\int_{\mathbb{Z}^d}|u_k|^{p-2}u_k(u_k-u)\log u_k^2\,d\mu\right|\nonumber\\&\rightarrow 0,\quad k\rightarrow\infty,  
\end{align*}
where $o_k(1)\rightarrow 0$ as $k\rightarrow\infty$. Furthermore, note that $u_k\rightharpoonup u$ in $H^{s,2}$, we have
$$\langle u,u_k-u\rangle_{H^{s,2}}\rightarrow 0,\quad k\rightarrow\infty.$$
Therefore, we deduce that
$$\|u_k-u\|_{H^{s,2}}\rightarrow 0,\quad k\rightarrow\infty.$$
Since $u_k\rightarrow u$ pointwise in $\mathbb{Z}^d$, we get $u_k\rightarrow u$ in $H^{s,2}.$ The proof is completed.

\end{proof}
\begin{lm}\label{lc}
 Let $p>4$ and $t\in [0,1)\cup(1,\infty).$  Then for any $u\in \mathcal{N}_{s,2}$, we have
 $$
   J_{s,2}(u)>J_{s,2}(tu).
 $$
\end{lm}

\begin{proof}
     A direct calculation gives us that
\begin{align*}
&J_{s,2}(u)-J_{s,2}(t u)\\= &\frac{1}{2}\left(\|u\|_{H^{s,2}}^{2}-\|t u\|_{H^{s,2}}^{2}\right)+\frac{b}{4}\left[\left(\int_{\mathbb{Z}^d}|\nabla^s u|^2\,d\mu\right)^2-\left(\int_{\mathbb{Z}^d}|\nabla^s (tu)|^2\,d\mu\right)^2\right]\\ &+\frac{2}{p^{2}} \int_{\mathbb{Z}^d}\left(|u|^{p}-|t u|^{p}\right) d \mu-\frac{1}{p} \int_{\mathbb{Z}^d}\left(|u|^{p} \log u^{2}-|t u|^{p} \log |t u|^{2}\right) d\mu\\
= & \frac{1-t^{2}}{2}\|u\|_{H^{s,2}}^{2}+\frac{b(1-t^4)}{4}\left(\int_{\mathbb{Z}^d}|\nabla^s u|^2\,d\mu\right)^2+\frac{2\left(1-t^{p}\right)}{p^{2}} \int_{\mathbb{Z}^d}|u|^{p} d \mu \\
& -\frac{1-t^{p}}{p} \int_{\mathbb{Z}^d}|u|^{p} \log u^{2} d \mu+\frac{t^{p}}{p} \int_{\mathbb{Z}^d}|u|^{p} \log t^{2} d \mu \\
= & \frac{1-t^{2}}{2}\|u\|_{H^{s,2}}^{2}+\frac{b(1-t^4)}{4}\left(\int_{\mathbb{Z}^d}|\nabla^s u|^2\,d\mu\right)^2+\left(\frac{2\left(1-t^{p}\right)}{p^{2}}+\frac{t^{p} \log t^{2}}{p}\right) \int_{\mathbb{Z}^d}|u|^{p} d \mu \\
& -\frac{1-t^{p}}{p} \int_{\mathbb{Z}^d}|u|^{p} \log u^{2} d \mu\\ \geq & \frac{1-t^{2}}{2}\|u\|_{H^{s,2}}^{2}+\frac{b(1-t^4)}{4}\left(\int_{\mathbb{Z}^d}|\nabla^s u|^2\,d\mu\right)^2 -\frac{1-t^{p}}{p} \int_{\mathbb{Z}^d}|u|^{p} \log u^{2} d \mu\\ =& \frac{1-t^{p}}{p}\left\langle J_{s,2}^{\prime}(u), u\right\rangle+\left(\frac{1-t^{2}}{2}-\frac{1-t^{p}}{p}\right)\|u\|_{H^{s,2}}^{2}+ b\left(\frac{1-t^4}{4}-\frac{1-t^p}{p}\right)\left(\int_{\mathbb{Z}^d}|\nabla^s u|^2\,d\mu\right)^2\\=& \left(\frac{1-t^{2}}{2}-\frac{1-t^{p}}{p}\right)\|u\|_{H^{s,2}}^{2}+ b\left(\frac{1-t^4}{4}-\frac{1-t^p}{p}\right)\left(\int_{\mathbb{Z}^d}|\nabla^s u|^2\,d\mu\right)^2\\ >& 0,
\end{align*}
where we have used $
2\left(1-t^{p}\right)+p t^{p} \log t^{2} \geq 0
$ for any $t\geq 0$ in the first inequality, and for $t\neq 1$,
$$\frac{1-t^{2}}{2}-\frac{1-t^{p}}{p}> 0,
$$
and
$$\frac{1-t^4}{4}-\frac{1-t^p}{p}>0$$ in the last inequality.

\end{proof}
The above lemma has a corollary, which says that for $u \in \mathcal{N}_{s,2}$ and $t \geq 0, J_{s,2}(t u)$ reaches its maximum at $t=1$.
\begin{crl}\label{lx}
Let $p>4$. For any $u \in \mathcal{N}_{s,2}$, we have that $$J_{s,2}(u)=\sup\limits_{t \geq 0} J_{s,2}(tu).$$
\end{crl}

The following lemma states that for any $u \in H^{s,2}\backslash\{0\}$ and $t>0$, $tu$ passes through the Nehari manifold once and only once, which also implies that $\mathcal{N}_{s,2}$ is non-empty.
\begin{lm}\label{lz}
Let $p>4$. For any $u \in H^{s,2}\backslash\{0\}$, there exists a unique $t_{u}>0$ such that $t_{u} u \in \mathcal{N}_{s,2}$.    
\end{lm}
\begin{proof}
For $u \in H^{s,2}\backslash\{0\}$ and $t>0$, similar to (\ref{98}), we obtain that
$$
\frac{1}{p}\int_{\mathbb{Z}^d}|tu|^{p} \log |tu|^{2} d \mu\leq \varepsilon t^2\|u\|^2_{H^{s,2}}+C_\varepsilon t^q\|u\|^q_{H_{s,2}}.
$$
Let $\varepsilon=\frac{1}{4},$ then we have that
$$
\begin{aligned}
 J_{s,2}(tu) 
& =\frac{t^{2}}{2}\|u\|_{H^{s,2}}^{2}+\frac{bt^4}{4}\left(\int_{\mathbb{Z}^d}|\nabla^s u|^2\,d\mu\right)^2+\frac{2}{p^2}\int_{\mathbb{Z}^d}|tu|^p\,d\mu-\frac{1}{p}\int_{\mathbb{Z}^d}|t u|^{p} \log |t u|^{2} d \mu\\
& \geq \frac{t^{2}}{2}\|u\|_{H^{s,2}}^{2}+\frac{bt^4}{4}\left(\int_{\mathbb{Z}^d}|\nabla^s u|^2\,d\mu\right)^2+\frac{2}{p^2}\int_{\mathbb{Z}^d}|tu|^p\,d\mu-\frac{t^{2}}{4}\|u\|_{H^{s,2}}^{2}-Ct^{q} \|u\|_{H^{s,2}}^{q} \\
& \geq \frac{t^{2}}{4}\|u\|_{H^{s,2}}^{2}-Ct^{q} \|u\|_{H^{s,2}}^{q}.   
\end{aligned}
$$
Since $q>4$, one gets easily that $J_{s,2}(tu)>0$ for $t>0$ small enough.

On the other hand, for any $u\in H^{s,2}\backslash\{0\}$ and $t>0$, by the fact $p>4$, it follows from (\ref{23}) that
$$
\begin{aligned}
    J_{s,2}(tu) 
& =\frac{t^{2}}{2}\|u\|_{H^{s,2}}^{2}+\frac{bt^4}{4}\left(\int_{\mathbb{Z}^d}|\nabla^s u|^2\,d\mu\right)^2+\frac{2}{p^2}\int_{\mathbb{Z}^d}|tu|^p\,d\mu-\frac{1}{p}\int_{\mathbb{Z}^d}|t u|^{p} \log |t u|^{2} d \mu\\ & =\frac{t^{2}}{2}\|u\|_{H^{s,2}}^{2}+\frac{bt^4}{4}\left(\int_{\mathbb{Z}^d}|\nabla^s u|^2\,d\mu\right)^2-\frac{t^p}{p^2}\left[\left(p\log t^2-2\right)\int_{\mathbb{Z}^d}|u|^p\,d\mu+p\int_{\mathbb{Z}^d}|u|^{p} \log |u|^{2} d \mu\right]
\\ &\rightarrow- \infty, \quad t\rightarrow \infty.
\end{aligned}
$$
Therefore, $\max\limits_{t>0} J_{s,2}(tu)$ is achieved at some $t_{u}>0$, and thus $t_{u} u \in \mathcal{N}_{s,2}$.

In the following, we prove the uniqueness of $t_u$. By contradiction, if there exist $t'_u>t_u>0$ such that $t'_u u\in\mathcal{N}_{s,2}$ and $t_u u\in\mathcal{N}_{s,2}$. Let $t=\frac{t'_u}{t_u}$, then $t\neq 1$.
By Lemma \ref{lc}, we have that
\begin{align*}
J_{s,2}\left(t'_u u\right) =J_{s,2}(t t_u u)<J_{s,2}(t_uu)
\end{align*}
and
\begin{align*}
J_{s,2}\left(t_u u\right)=J_{s,2}\left(\frac{1}{t}t'_u u\right)<J_{s,2}(t'_uu),
\end{align*}
which is impossible. Hence there exists a unique $t_{u}>0$ such that $t_uu\in \mathcal{N}_{s,2}$.

\end{proof}

By Lemma \ref{lm} and Lemma \ref{lb},  we know that $J_{s,2}$ satisfies the geometric structure of the
mountain pass theorem. Hence we derive that $$c=\inf\limits_{\gamma\in\Gamma}\sup\limits_{t\in [0,1]} J_{s,2}(\gamma(t))$$ is a critical level of $J_{s,2}$, where

$$
\Gamma=\left\{\gamma \in C\left([0,1], H^{s,2}\right): \gamma(0)=0, J_{s,2}(\gamma(1))<0\right\} .
$$ In particular, there exists $u\in H^{s,2}$ such that $J_{s,2}(u)=c\geq\sigma>0$, which implies that the equation (\ref{0.2}) has a nontrivial mountain pass solution.
 
Denote $$\hat{c}=\inf_{u \in \mathcal{N}_{s,2}} J_{s,2}(u).$$
We show that  the least energy level $\hat{c}=\inf\limits_{u \in \mathcal{N}_{s,2}} J_{s,2}(u)$ is equal to the mountain pass level $c=\inf\limits_{\gamma \in \Gamma} \sup\limits_{t \in[0,1]} J_{s,2}(\gamma(t))$.

\begin{lm}\label{kq}
We have $\hat{c}=c>0.$   
\end{lm}

\begin{proof}
 Clearly, the mountain pass solution $u$ belongs to $\mathcal{N}_{s,2}$, and hence we have $\hat{c} \leq c$. In order to prove $\hat{c} \geq c$, we define 
$$
\bar{c}=\inf _{u \in H^{s,2} \backslash\{0\}} \sup _{t \geq 0} J_{s,2}(t u) .
$$
We first show that $\hat{c}=\bar{c}$. By Corollary \ref{lx} and Lemma \ref{lz}, for any $u \in H^{s,2} \backslash\{0\}$, there exists a unique $t_{u}>0$ such that $t_uu\in\mathcal{N}_{s,2}$ and
$$
J_{s,2}\left(t_{u} u\right)=\sup_{t \geq 0} J_{s,2}(t u),
$$
which implies that
$$
\hat{c}=\inf _{u \in \mathcal{N}_{s,2}} J_{s,2}(u)=\inf _{u \in H^{s,2} \backslash\{0\}} J_{s,2}\left(t_{u} u\right)=\inf _{u \in H^{s,2} \backslash\{0\}} \sup _{t \geq 0} J_{s,2}(t u)=\bar{c} .
$$

Now we demonstrate that $\bar{c} \geq c$.  By the proof of (ii) in Lemma \ref{lm}, for any $u \in H^{s,2} \backslash\{0\}$, there exists $t_{0}>0$ large enough such that $J_{s,2}\left(t_{0} u\right)<0$. Define
$$
\gamma_0: t \in[0,1] \rightarrow t t_{0} u \in H^{s,2}.
$$
It is obvious that $\gamma_0(0)=0$ and $J_{s,2}\left(\gamma_0(1)\right)=J_{s,2}\left(t_{0} u\right)<0$. So it is a path in
$$
\Gamma=\left\{\gamma \in C\left([0,1], H^{s,2}\right): \gamma(0)=0, J_{s,2}(\gamma(1))<0\right\} .
$$
Therefore, for any $u \in H^{s,2} \backslash\{0\}$, we obtain that
$$
\sup _{t \geq 0} J_{s,2}(t u) \geq \sup _{t \in[0,1]} J_{s,2}\left(t t_{0} u\right)=\sup _{t \in[0,1]} J_{s,2}\left(\gamma_0(t)\right) \geq \inf _{\gamma \in \Gamma} \sup_{t \in[0,1]} J_{s,2}(\gamma(t)),
$$
which yields that
$$
\hat{c}=\bar{c}=\inf _{u \in H^{s,2} \backslash\{0\}} \sup _{t \geq 0} J(t u) \geq \inf _{\gamma \in \Gamma} \sup _{t \in[0,1]} J_{s,2}(\gamma(t))=c.
$$
   
\end{proof}

{\bf Proof of Theorem \ref{t1}.} By Lemma \ref{lm} and Lemma \ref{lb}, one sees that $J_{s,2}$ satisfies the geometric structure and $(PS)_{c}$ condition. Then by the mountain pass theorem, there exists $u\in H^{s,2}$ such that $J_{s,2}(u)=c$ and $J_{s,2}'(u)=0$. Then it follows from Lemma \ref{kq} that $c=\hat{c}>0$. Hence $u\neq 0$ and $u\in \mathcal{N}_{s,2}.$ The proof is completed. \qed

\section{Ground state sign-changing solutions}
In this section, we are devoted to proving Theorem \ref{t2} and Theorem \ref{t3}. We start our proofs with some auxiliary lemmas.

\begin{lm}\label{l3}
 Let $p> 6$. For any $u \in \mathcal{M}_{s,2}$ and $(r, t) \in(0, \infty) \times(0, \infty)$ with $(r,t)\neq(1,1)$, we have
$$
\begin{aligned}
J_{s,2}(u) >  J_{s,2}\left(r u^{+}+t u^{-}\right).
\end{aligned}
$$

\end{lm}

\begin{proof}

For any $u \in \mathcal{M}_{s,2}$ and $r, t>0$, by Proposition \ref{o}, we obtain that
\begin{align*}
& J_{s,2}(u)-J_{s,2}\left(r u^{+}+t u^{-}\right) \\
= & \frac{1}{2}\left(\|u\|_{H^{s,2}}^{2}-\left\|r u^{+}+t u^{-}\right\|_{H^{s,2}}^{2}\right)+\frac{b}{4}\left(\|\nabla^s u\|_{2}^{4}-\left\|\nabla^s\left(r u^{+}+t u^{-}\right)\right\|_{2}^{4}\right) \\
& +\frac{2}{p^{2}} \int_{\mathbb{Z}^d}\left[|u|^{p}-\left|r u^{+}+t u^{-}\right|^{p}\right] \,d\mu-\frac{1}{p} \int_{\mathbb{Z}^d}\left[|u|^{p} \log u^{2}-\left|r u^{+}+t u^{-}\right|^{p} \log \left(r u^{+}+t u^{-}\right)^{2}\right] \,d\mu.
\\
= & \frac{1-r^2}{2}\left\|u^{+}\right\|_{H^{s,2}}^{2}+\frac{1-t^{2}}{2}\left\|u^{-}\right\|_{H^{s,2}}^{2}-\frac{a(1-rt)}{2}K(u)+\frac{b\left(1-r^{4}\right)}{4}\left\|\nabla^s u^{+}\right\|_{2}^{4}+\frac{b\left(1-t^{4}\right)}{4}\left\|\nabla^s u^{-}\right\|_{2}^{4} \\
& +\frac{b\left(1-r^2 t^{2}\right)}{2}\left\|\nabla^s u^{+}\right\|_{2}^{2}\left\|\nabla^s u^{-}\right\|_{2}^{2}-\frac{b(1-r^3t)}{2}K(u)\left\|\nabla^s u^{+}\right\|_{2}^{2}-\frac{b(1-rt^3)}{2}K(u)\left\|\nabla^s u^{-}\right\|_{2}^{2}\\&+\frac{b(1-r^2t^2)}{4}K^2(u)+\frac{2}{p^{2}} \int_{\mathbb{Z}^d}\left[\left|u^{+}\right|^{p}-\left|r u^{+}\right|^{p}+\left|u^{-}\right|^{p}-\left|t u^{-}\right|^{p}\right] \,d\mu\\
& -\frac{1}{p} \int_{\mathbb{Z}^d}\left[\left|u^{+}\right|^{p} \log \left(u^{+}\right)^{2}-\left|r u^{+}\right|^{p} \log \left(u^{+}\right)^{2}-\left|r u^{+}\right|^{p} \log r^2\right] \,d\mu\\
& -\frac{1}{p} \int_{\mathbb{Z}^d}\left[\left|u^{-}\right|^{p} \log \left(u^{-}\right)^{2}-\left|t u^{-}\right|^{p} \log \left(u^{-}\right)^{2}-\left|t u^{-}\right|^{p} \log t^{2}\right] \,d\mu\\
= & \frac{1-r^{p}}{p}\left\langle J_{s,2}^{\prime}(u), u^{+}\right\rangle+\frac{1-t^{p}}{p}\left\langle J_{s,2}^{\prime}(u), u^{-}\right\rangle+\left(\frac{1-r^2}{2}-\frac{1-r^{p}}{p}\right)\left\|u^{+}\right\|_{H^{s,2}}^{2} \\
& +\left(\frac{1-t^{2}}{2}-\frac{1-t^{p}}{p}\right)\left\|u^{-}\right\|_{H^{s,2}}^{2}+b\left[\left(\frac{1-r^{4}}{4}-\frac{1-r^{p}}{p}\right)\left\|\nabla^s u^{+}\right\|_{2}^{4}\right. \\
& \left.+\left(\frac{1-t^{4}}{4}-\frac{1-t^{p}}{p}\right)\left\|\nabla^s u^{-}\right\|_{2}^{4}+\left(\frac{1-r^2 t^{2}}{2}-\frac{1-r^{p}}{p}-\frac{1-t^{p}}{p}\right)\left\|\nabla^s u^{+}\right\|_{2}^{2}\left\|\nabla^s u^{-}\right\|_{2}^{2}\right]\\&+\frac{b}{2}\left(\frac{1-r^2 t^{2}}{2}-\frac{1-r^{p}}{p}-\frac{1-t^{p}}{p}\right)K^2(u)-\frac{a}{2}\left((1-rt)-\frac{1-r^{p}}{p}-\frac{1-t^{p}}{p}\right)K(u)\\&-\frac{b}{2}\left((1-r^3t)-\frac{3(1-r^{p})}{p}-\frac{1-t^{p}}{p}\right)K(u)\|\nabla^s u^+\|^2_2\\&-\frac{b}{2}\left((1-rt^3)-\frac{1-r^{p}}{p}-\frac{3(1-t^{p})}{p}\right)K(u)\|\nabla^s u^-\|^2_2\\
& +\frac{2\left(1-r^{p}\right)+p r^{p} \log r^2}{p^{2}} \int_{\mathbb{Z}^d}\left|u^{+}\right|^{p} \,d\mu+\frac{2\left(1-t^{p}\right)+p t^{p} \log t^{2}}{p^{2}} \int_{\mathbb{Z}^d}\left|u^{-}\right|^{p} \,d\mu
\\
\geq & \frac{1-r^{p}}{p}\left\langle J_{s,2}^{\prime}(u), u^{+}\right\rangle+\frac{1-t^{p}}{p}\left\langle J_{s,2}^{\prime}(u), u^{-}\right\rangle+\left(\frac{1-r^2}{2}-\frac{1-r^{p}}{p}\right)\left\|u^{+}\right\|_{H^{s,2}}^{2} \\
& +\left(\frac{1-t^{2}}{2}-\frac{1-t^{p}}{p}\right)\left\|u^{-}\right\|_{H^{s,2}}^{2}+b\left[\left(\frac{1-r^{4}}{4}-\frac{1-r^{p}}{p}\right)\left\|\nabla^s u^{+}\right\|_{2}^{4}\right. \\
& \left.+\left(\frac{1-t^{4}}{4}-\frac{1-t^{p}}{p}\right)\left\|\nabla^s u^{-}\right\|_{2}^{4}+\left(\frac{1-r^2 t^{2}}{2}-\frac{1-r^{p}}{p}-\frac{1-t^{p}}{p}\right)\left\|\nabla^s u^{+}\right\|_{2}^{2}\left\|\nabla^s u^{-}\right\|_{2}^{2}\right]\\&+\frac{b}{2}\left(\frac{1-r^2 t^{2}}{2}-\frac{1-r^{p}}{p}-\frac{1-t^{p}}{p}\right)K^2(u)-\frac{a}{2}\left((1-rt)-\frac{1-r^{p}}{p}-\frac{1-t^{p}}{p}\right)K(u)\\&-\frac{b}{2}\left((1-r^3t)-\frac{3(1-r^{p})}{p}-\frac{1-t^{p}}{p}\right)K(u)\|\nabla^s u^+\|^2_2\\&-\frac{b}{2}\left((1-rt^3)-\frac{1-r^{p}}{p}-\frac{3(1-t^{p})}{p}\right)K(u)\|\nabla^s u^-\|^2_2\\ =&\left(\frac{1-r^2}{2}-\frac{1-r^{p}}{p}\right)\left\|u^{+}\right\|_{H^{s,2}}^{2} +\left(\frac{1-t^{2}}{2}-\frac{1-t^{p}}{p}\right)\left\|u^{-}\right\|_{H^{s,2}}^{2}\\&+b\left[\left(\frac{1-r^{4}}{4}-\frac{1-r^{p}}{p}\right)\left\|\nabla^s u^{+}\right\|_{2}^{4}+\left(\frac{1-t^{4}}{4}-\frac{1-t^{p}}{p}\right)\left\|\nabla^s u^{-}\right\|_{2}^{4}\right]\\ &+b\left[\frac{\left(r^2-t^{2}\right)^{2}}{4}+\left(\frac{1-r^{4}}{4}-\frac{1-r^{p}}{p}\right)+\left(\frac{1-t^{4}}{4}-\frac{1-t^{p}}{p}\right)\right]\left\|\nabla^s u^{+}\right\|_{2}^{2}\left\|\nabla^s u^{-}\right\|_{2}^{2}\\&+\frac{b}{2}\left[\frac{\left(r^2-t^{2}\right)^{2}}{4}+\left(\frac{1-r^{4}}{4}-\frac{1-r^{p}}{p}\right)+\left(\frac{1-t^{4}}{4}-\frac{1-t^{p}}{p}\right)\right]K^2(u)\\&-\frac{a}{2}\left[\frac{\left(r-t\right)^{2}}{2}+\left(\frac{1-r^2}{2}-\frac{1-r^{p}}{p}\right)+\left(\frac{1-t^{2}}{2}-\frac{1-t^{p}}{p}\right)\right]K(u)\\&-\frac{b}{2}\left[\frac{\left(r^{3}-t\right)^{2}}{2}+3\left(\frac{1-r^{6}}{6}-\frac{1-r^{p}}{p}\right)+\left(\frac{1-t^{2}}{2}-\frac{1-t^{p}}{p}\right)\right]K(u)\|\nabla^s u^+\|^2_2\\&-\frac{b}{2}\left[\frac{\left(r-t^{3}\right)^{2}}{2}+\left(\frac{1-r^2}{2}-\frac{1-r^{p}}{p}\right)+3\left(\frac{1-t^{6}}{6}-\frac{1-t^{p}}{p}\right)\right]K(u)\|\nabla^s u^-\|^2_2\\> &0,
\end{align*}
where we have used the facts that 
$$2\left(1-x^{p}\right)+p x^{p} \log x^{2}\geq 0, \quad x \in(0,\infty)$$  in the first inequality, and $p>6$, $K(u)<0$,
$$g(x)=\frac{1-a^{x}}{x}$$ is a strictly decreasing function in $(0,1)\cup(1,\infty)$ in the last inequality. The proof is completed.

\end{proof}

\begin{lm}\label{4}
    
If $u \in H^{s,2}$ with $u^{ \pm} \neq 0$, then there exists a unique pair $\left(r_{u}, t_{u}\right)\in (0,\infty)\times(0,\infty)$ such that $r_{u} u^{+}+t_{u} u^{-} \in \mathcal{M}_{s,2}$.
\end{lm}

\begin{proof}
    
For $r, t>0$, we denote
\begin{eqnarray*}
\varphi_{1}(r, t) = \left\langle J_{s,2}^{\prime}\left(r u^{+}+t u^{-}\right),ru^{+}\right\rangle,\qquad \varphi_{2}(r, t)=\left\langle J_{s,2}^{\prime}\left(r u^{+}+t u^{-}\right),t u^{-}\right\rangle.  
\end{eqnarray*}
By Proposition \ref{o}, we obtain that
\begin{equation}\label{2.7}
\begin{aligned}
\varphi_{1}(r, t) =&\left\langle J_{s,2}^{\prime}\left(r u^{+}+t u^{-}\right),ru^{+}\right\rangle\\=& r^2\|u^{+}\|_{H^{s,2}}^{2}+b r^{4}\left\|\nabla^s u^{+}\right\|_{2}^{4}-\int_{\mathbb{Z}^d}\left|r u^{+}\right|^{p} \log \left(r u^{+}\right)^{2}\,d\mu-\frac{a}{2}rtK(u)+\frac{b}{2}r^2t^2K^2(u) 
\\&+b r^2 t^{2}\left\|\nabla^s u^{+}\right\|_{2}^{2}\left\|\nabla^s u^{-}\right\|_{2}^{2}-\frac{b}{2}rtK(u)\left(r^2\left\|\nabla^s u^{+}\right\|_{2}^{2}+t^2\left\|\nabla^s u^{-}\right\|_{2}^{2}\right)
-br^3tK(u)\|\nabla^s u^{+}\|^2_2,
\end{aligned}
\end{equation}
and
\begin{equation}\label{2.8}
\begin{aligned}
\varphi_{2}(r, t) =&\left\langle J_{s,2}^{\prime}\left(r u^{+}+t u^{-}\right),t u^{-}\right\rangle\\=& t^{2}\|u^{-}\|_{H^{s,2}}^{2}+b t^{4}\left\|\nabla^s u^{-}\right\|_{2}^{4}-\int_{\mathbb{Z}^d}\left|t u^{-}\right|^{p} \log \left(t u^{-}\right)^{2}\,d\mu-\frac{a}{2}rtK(u)+\frac{b}{2}r^2t^2K^2(u) 
\\&+b r^2 t^{2}\left\|\nabla^s u^{+}\right\|_{2}^{2}\left\|\nabla^s u^{-}\right\|_{2}^{2}-\frac{b}{2}rtK(u)\left(r^2\left\|\nabla^s u^{+}\right\|_{2}^{2}+t^2\left\|\nabla^s u^{-}\right\|_{2}^{2}\right)
-brt^3K(u)\|\nabla^s u^{-}\|^2_2.
\end{aligned}
\end{equation}
Let $t=r$ in (\ref{2.7}) and (\ref{2.8}), then
$$
\begin{aligned}
\varphi_{1}(r, r) =& r^2\|u^{+}\|_{H^{s,2}}^{2}+b r^{4}\left\|\nabla^s u^{+}\right\|_{2}^{4}-r^p\int_{\mathbb{Z}^d}\left| u^{+}\right|^{p} \log \left( u^{+}\right)^{2}\,d\mu-r^p\log r^2\int_{\mathbb{Z}^d}\left|u^{+}\right|^{p}\,d\mu-\frac{a}{2}r^2K(u)\\&+\frac{b}{2}r^4K^2(u) 
+b r^{4} \left\|\nabla^s u^{+}\right\|_{2}^{2}\left\|\nabla^s u^{-}\right\|_{2}^{2}-\frac{b}{2}r^4K(u)\left(\left\|\nabla^s u^{+}\right\|_{2}^{2}+\left\|\nabla^s u^{-}\right\|_{2}^{2}\right)
-br^4K(u)\|\nabla^s u^{+}\|^2_2,
\end{aligned}
$$
and
$$
\begin{aligned}
\varphi_{2}(r, r) =& r^2\|u^{-}\|_{H^{s,2}}^{2}+b r^{4}\left\|\nabla^s u^{-}\right\|_{2}^{4}-r^p\int_{\mathbb{Z}^d}\left| u^{-}\right|^{p} \log \left( u^{-}\right)^{2}\,d\mu-r^p\log r^2\int_{\mathbb{Z}^d}\left|u^{-}\right|^{p}\,d\mu-\frac{a}{2}r^2K(u)\\&+\frac{b}{2}r^4K^2(u) 
+b r^{4} \left\|\nabla^s u^{+}\right\|_{2}^{2}\left\|\nabla^s u^{-}\right\|_{2}^{2}-\frac{b}{2}r^4K(u)\left(\left\|\nabla^s u^{+}\right\|_{2}^{2}+\left\|\nabla^s u^{-}\right\|_{2}^{2}\right)
-br^4K(u)\|\nabla^s u^{-}\|^2_2.
\end{aligned}
$$
Since $p>6$ and $K(u)< 0$, we have that $\varphi_{1}(r, r)>0$ and $\varphi_{2}(r, r)>0$ for $r>0$ small enough and $\varphi_{1}(r, r)<0$ and $\varphi_{2}(r, r)<0$ for $r>0$ large enough. Therefore, there exist $r_1$ and $R_1$ with $0<r_1<R_1$ such that
\begin{equation}\label{0.5}
\varphi_{1}(r_1, r_1)>0,\quad \varphi_{2}(r_1, r_1)>0,\quad \varphi_{1}(R_1, R_1)<0,\quad \varphi_{2}(R_1, R_1)<0.
\end{equation}
By (\ref{2.7}), (\ref{2.8}) and (\ref{0.5}), we obtain that
$$
\begin{aligned}
&\varphi_{1}(r_1, t)>0, \qquad\varphi_{1}(R_1, t)<0,\quad t \in[r_1, R_1], \\
& \varphi_{2}(r, r_1)>0,\qquad \varphi_{2}(r, R_1)<0,\quad r\in[r_1, R_1].
\end{aligned}
$$
By the Miranda's theorem \cite{K}, there exist $r_{u}, t_{u} \in(r_1, R_1)$ such that $\varphi_{1}\left(r_{u}, t_{u}\right)=\varphi_{2}\left(r_{u}, t_{u}\right)=0$, which implies that $r_{u} u^{+}+t_{u} u^{-} \in \mathcal{M}_{s,2}$.

Now we verify the uniqueness of the pair $\left(r_{u}, t_{u}\right)$. Otherwise, there exist $\left(r_1, t_{1}\right)$ and $\left(r_{2}, t_{2}\right)$ with $r_1\neq r_2$ and $t_1\neq t_2$ such that $r_{1} u^{+}+t_{1} u^{-} \in \mathcal{M}_{s,2}$ and $r_{2} u^{+}+t_{2} u^{-} \in \mathcal{M}_{s,2}$.

Let $r=\frac{r_2}{r_1}$ and $t=\frac{t_2}{t_1}$. Clearly $r\neq 1$ and $t\neq 1$. Then it follows from Lemma \ref{l3} that
$$
J_{s,2}(r_2u^++t_2u^-) =J_{s,2} \left(r(r_1u^+)+t(t_1u^-)\right)<J_{s,2} (r_1u^++t_1u^-),
$$
and
$$
  J_{s,2}(r_1u^++t_1u^-) =J_{s,2} \left(\frac{1}{r}(r_2u^+)+\frac{1}{t}(t_2u^-)\right)<J_{s,2} (r_2u^++t_2u^-).  
$$
This is impossible. Thus there exists a unique pair $(r_u, t_u)\in (0,\infty)\times(0,\infty)$ such that $r_uu^++t_uu^-\in\mathcal{M}_{s,2}.$

\end{proof}

\begin{lm}\label{l5}
Let $u \in H^{s,2}$ with $u^{ \pm} \neq 0$ such that $\langle J_{s,2}^{\prime}(u),u^{ \pm}\rangle \leq 0$. Then the unique pair $\left(r_{u}, t_{u}\right)$ appeared in Lemma \ref{4} satisfies $r_{u}, t_{u} \in(0,1]$.
\end{lm}

\begin{proof}
By Lemma \ref{4}, there exists a unique pair $(r_u,t_u)\in(0,\infty)\times(0,\infty)$ such that $r_uu^++t_uu^-\in \mathcal{M}_{s,2}$. Without loss of generality, we assume that $0<t_{u} \leq r_{u}$. Since $r_{u} u^{+}+t_{u} u^{-} \in \mathcal{M}_{s,2}$, we have that
$$
\begin{aligned}
&\langle J_{s,2}^{\prime}(r_uu^++t_uu^-),r_uu^{+}\rangle\\=&r_u^{2}\|u^{+}\|_{H^{s,2}}^{2}+b r_u^{4}\left\|\nabla^s u^{+}\right\|_{2}^{4}-\int_{\mathbb{Z}^d}\left|r_u u^{+}\right|^{p} \log \left(r_u u^{+}\right)^{2}\,d\mu-\frac{a}{2}r_ut_uK(u)+\frac{b}{2}r_u^2t_u^2K^2(u)
\\&+b r_u^{2} t_u^{2}\left\|\nabla^s u^{+}\right\|_{2}^{2}\left\|\nabla^s u^{-}\right\|_{2}^{2}-\frac{b}{2}r_ut_uK(u)\left(r_u^2\left\|\nabla^s u^{+}\right\|_{2}^{2}+t_u^2\left\|\nabla^s u^{-}\right\|_{2}^{2}\right)
-br_u^3t_uK(u)\|\nabla^s u^{+}\|^2_2\\=& 0.
\end{aligned}
$$
Namely
\begin{align}\label{28}
\int_{\mathbb{Z}^d}\left|r_u u^{+}\right|^{p} \log \left(r_u u^{+}\right)^{2}\,d\mu=\nonumber&r_u^{2}\|u^{+}\|_{H^{s,2}}^{2}+b r_u^{4}\left\|\nabla^s u^{+}\right\|_{2}^{4}-\frac{a}{2}r_ut_uK(u)+\frac{b}{2}r_u^2t_u^2K^2(u)
\\&\nonumber+b r_u^{2} t_u^{2}\left\|\nabla^s u^{+}\right\|_{2}^{2}\left\|\nabla^s u^{-}\right\|_{2}^{2}-\frac{b}{2}r_ut_uK(u)\left(r_u^2\left\|\nabla^s u^{+}\right\|_{2}^{2}+t_u^2\left\|\nabla^s u^{-}\right\|_{2}^{2}\right)
\\&\nonumber-br_u^3t_uK(u)\|\nabla^s u^{+}\|^2_2\\ \leq &r_u^{2}\|u^{+}\|_{H^{s,2}}^{2}+b r_u^{4}\left\|\nabla^s u^{+}\right\|_{2}^{4}-\frac{a}{2}r^2_uK(u)+\frac{b}{2}r_u^4K^2(u)
\\&\nonumber+b r_u^{4} \left\|\nabla^s u^{+}\right\|_{2}^{2}\left\|\nabla^s u^{-}\right\|_{2}^{2}-\frac{b}{2}r^4_uK(u)\left(\left\|\nabla^s u^{+}\right\|_{2}^{2}+\left\|\nabla^s u^{-}\right\|_{2}^{2}\right)
\\&\nonumber-br_u^4K(u)\|\nabla^s u^{+}\|^2_2,
\end{align}    
where we have used the fact that $K(u)<0$.
Moreover, we have that
$$
\begin{aligned}
\langle J_{s,2}^{\prime}(u),u^{+}\rangle=&\|u^{+}\|_{H^{s,2}}^{2}+b\left\|\nabla^s u^{+}\right\|_{2}^{4}-\int_{\mathbb{Z}^d}\left|u^{+}\right|^{p} \log \left(u^{+}\right)^{2}\,d\mu-\frac{a}{2}K(u)+\frac{b}{2}K^2(u)
\\&+b\left\|\nabla^s u^{+}\right\|_{2}^{2}\left\|\nabla^s u^{-}\right\|_{2}^{2}-\frac{b}{2}K(u)\left(\left\|\nabla^s u^{+}\right\|_{2}^{2}+\left\|\nabla^s u^{-}\right\|_{2}^{2}\right)
-bK(u)\|\nabla^s u^{+}\|^2_2\\ \leq &0.
\end{aligned}
$$
As a consequence,
$$
\begin{aligned}
\int_{\mathbb{Z}^d}\left|u^{+}\right|^{p} \log \left(u^{+}\right)^{2}\,d\mu\geq &\|u^{+}\|_{H^{s,2}}^{2}+b\left\|\nabla^s u^{+}\right\|_{2}^{4}-\frac{a}{2}K(u)+\frac{b}{2}K^2(u)
+b\left\|\nabla^s u^{+}\right\|_{2}^{2}\left\|\nabla^s u^{-}\right\|_{2}^{2}\\&-\frac{b}{2}K(u)\left(\left\|\nabla^s u^{+}\right\|_{2}^{2}+\left\|\nabla^s u^{-}\right\|_{2}^{2}\right)
-bK(u)\|\nabla^s u^{+}\|^2_2.
\end{aligned}
$$
Multiplying the previous inequality by $-r_u^p$, then
\begin{equation}\label{3.5}
 \begin{aligned}
-\int_{\mathbb{Z}^d}\left|r_uu^{+}\right|^{p} \log \left(u^{+}\right)^{2}\,d\mu\leq &-r_u^p\|u^{+}\|_{H^{s,2}}^{2}-br_u^p\left\|\nabla^s u^{+}\right\|_{2}^{4}+\frac{a}{2}r_u^pK(u)-\frac{b}{2}r_u^pK^2(u)
\\&-br_u^p\left\|\nabla^s u^{+}\right\|_{2}^{2}\left\|\nabla^s u^{-}\right\|_{2}^{2}+\frac{b}{2}r_u^pK(u)\left(\left\|\nabla^s u^{+}\right\|_{2}^{2}+\left\|\nabla^s u^{-}\right\|_{2}^{2}\right)
\\&+br_u^pK(u)\|\nabla^s u^{+}\|^2_2.
\end{aligned}   
\end{equation}
We claim that $r_u\leq 1$. In fact, if $r_u>1$, then
by adding (\ref{28}) and (\ref{3.5}), we get that
$$
\begin{aligned}
0<&r_u^p\log r_u^{2}\int_{\mathbb{Z}^d}\left|u^{+}\right|^{p}\,d\mu\\\leq &(r_u^2-r_u^p)\|u^{+}\|_{H^{s,2}}^{2}+b(r_u^4-r_u^p)\left\|\nabla^s u^{+}\right\|_{2}^{4}-\frac{a}{2}(r_u^2-r_u^p)K(u)+\frac{b}{2}(r_u^4-r_u^p)K^2(u)
\\&+b(r_u^4-r_u^p)\left\|\nabla^s u^{+}\right\|_{2}^{2}\left\|\nabla^s u^{-}\right\|_{2}^{2}-\frac{b}{2}(r_u^4-r_u^p)K(u)\left(\left\|\nabla^s u^{+}\right\|_{2}^{2}+\left\|\nabla^s u^{-}\right\|_{2}^{2}\right)
\\&-b(r_u^4-r_u^p)K(u)\|\nabla^s u^{+}\|^2_2\\ <& 0,
\end{aligned} 
$$
which is a contradiction. Therefore we obtain that $0<t_u\leq r_u\leq 1.$ 
\end{proof}

Now we prove that the minimizer of $J_{s,2}$ on $\mathcal{M}_{s,2}$ can be achieved. We denote
$$m=\inf_{u \in \mathcal{M}_{s,2}} J_{s,2}(u).$$

\begin{lm}\label{l9}
Let $p>6$ and $(h_1)$-$(h_2)$ hold. Then there exists $u\in \mathcal{M}_{s,2}$ such that $J_{s,2}(u)=m>0$.
\end{lm}

\begin{proof} 
Let $\{u_k\}\subset\mathcal{M}_{s,2}$ be a sequence such that
$$\lim\limits_{k\rightarrow\infty}J_{s,2}(u_k)=m.$$
By Proposition \ref{o}, one gets easily that $$\langle J'_{s,2}(u_k),u_k\rangle=\langle J'_{s,2}(u_k),u_k^+\rangle+\langle J'_{s,2}(u_k),u_k^-\rangle =0.$$
Then we have that
$$
\begin{aligned}
\lim _{k \rightarrow\infty} J_{s,2}\left(u_{k}\right) & =\lim _{k \rightarrow\infty}\left[J_{s,2}\left(u_{k}\right)-\frac{1}{p} \langle J_{s,2}^{\prime}\left(u_{k}\right), u_{k}\rangle\right] \\
& =\lim _{k \rightarrow\infty}\left[\left(\frac{1}{2}-\frac{1}{p}\right)\|u_k\|_{H^{s,2}}^{2}+b\left(\frac{1}{4}-\frac{1}{p}\right)\left\|\nabla^s u_k\right\|_{2}^{4}+\frac{2}{p^2}\int_{\mathbb{Z}^d} |u_k|^p\,d\mu\right] \\
& =m, 
\end{aligned}
$$
which means that
$$\left(\frac{1}{2}-\frac{1}{p}\right)\|u_k\|_{H^{s,2}}^{2}\leq \left(\frac{1}{2}-\frac{1}{p}\right)\|u_k\|_{H^{s,2}}^{2}+b\left(\frac{1}{4}-\frac{1}{p}\right)\left\|\nabla^s u_k\right\|_{2}^{4}+\frac{2}{p^2}\int_{\mathbb{Z}^d} |u_k|^p\,d\mu\leq m+1.$$
Therefore, we get that $\left\{u_{k}\right\}$ is bounded in $H^{s,2}$. By Lemma \ref{lg}, there exists $u \in H^{s,2}$ such that
\begin{equation}\label{29}
\begin{cases}u_{k} \rightharpoonup u, & \text { weakly in } H^{s,2}, \\ u_{k} \rightarrow u, & \text { pointwise in } \mathbb{Z}^d, \\ u_{k} \rightarrow u, & \text { strongly in } \ell^{q}(\mathbb{Z}^d), ~q \in[2,\infty].\end{cases}
\end{equation}
By (\ref{1.2}), for any $\varepsilon>0$ and $t\neq 0,$ we have
\begin{equation}\label{3.9}
|t|^p|\log t^2|\leq (\varepsilon |t|^{2}+C_\varepsilon |t|^{q}),\quad q>p.
\end{equation}
Moreover, we have
\begin{equation}\label{3.8}
\langle J'_{s,2}(u_k),u_k^{\pm}\rangle=\|u_k^{\pm}\|_{H^{s,2}}^{2}+b\left\|\nabla^s u_k\right\|_{2}^{2}\left(\|\nabla^s u_k^{\pm}\|_{2}^{2}-\frac{1}{2}K(u_k)\right)-\frac{a}{2}K(u_k)-\int_{\mathbb{Z}^d}\left|u_k^{\pm}\right|^{p} \log \left(u_k^{\pm}\right)^{2}\,d\mu=0.
\end{equation}
Then we get that
$$
\begin{aligned}
\|u_k^{\pm}\|_{H^{s,2}}^{2}< &\|u_k^{\pm}\|_{H^{s,2}}^{2}+b\left\|\nabla^s u_k\right\|_{2}^{2}\left(\|\nabla^s u_k^{\pm}\|_{2}^{2}-\frac{1}{2}K(u_k)\right)-\frac{a}{2}K(u_k)\\=&\int_{\mathbb{Z}^d}\left|u_k^{\pm}\right|^{p} \log \left(u_k^{\pm}\right)^{2}\,d\mu\\\leq &\int_{\mathbb{Z}^d} \left(\varepsilon|u_k^{\pm}|^{2}+C_\varepsilon|u_k^{\pm}|^{q}\right)\,d\mu\\\leq &\varepsilon\|u^{\pm}\|^2_{H^{s,2}}+C_\varepsilon\|u^{\pm}\|^q_{H^{s,2}}.
\end{aligned}
$$
Since $p>6$, we deduce that
$\|u_k^{\pm}\|_{H^{s,2}}\geq C>0,$ and hence $u^{\pm}\neq 0$. Therefore, for $p>6$,  by (\ref{29}), we derive that
$$
\begin{aligned}
m=&\lim _{k \rightarrow\infty}\left[\left(\frac{1}{2}-\frac{1}{p}\right)\|u_k\|_{H^{s,2}}^{2}+b\left(\frac{1}{4}-\frac{1}{p}\right)\left\|\nabla^s u_k\right\|_{2}^{4}+\frac{2}{p^2}\int_{\mathbb{Z}^d} |u_k|^p\,d\mu\right]\\ \geq&
\lim _{k \rightarrow\infty}\frac{2}{p^2}\int_{\mathbb{Z}^d} |u_k|^p\,d\mu\\=&
\frac{2}{p^2}\int_{\mathbb{Z}^d} |u|^p\,d\mu\\\geq &
\frac{2}{p^2}\int_{\mathbb{Z}^d} |u^{\pm}|^p\,d\mu\\>& 0.
\end{aligned}
$$
Then it follows from (\ref{29}), (\ref{3.9}), (\ref{3.8}) and Lebesgue dominated theorem that
\begin{align*}
&\|u^{\pm}\|_{H^{s,2}}^{2}+b\left\|\nabla^s u\right\|_{2}^{2}\left(\|\nabla^s u^{\pm}\|_{2}^{2}-\frac{1}{2}K(u)\right)-\frac{a}{2}K(u)-\int_{\mathbb{Z}^d}\left(|u^{\pm}|^{p} \log \left(u^{\pm}\right)^{2}\right)^-\,d\mu \\
\leq & \limsup _{k \rightarrow\infty}\left[\|u_k^{\pm}\|_{H^{s,2}}^{2}+b\left\|\nabla^s u_k\right\|_{2}^{2}\left(\|\nabla^s u_k^{\pm}\|_{2}^{2}-\frac{1}{2}K(u_k)\right)-\frac{a}{2}K(u_k)-\int_{\mathbb{Z}^d}\left(|u_k^{\pm}|^{p} \log \left(u_k^{\pm}\right)^{2}\right)^-\,d\mu\right] \\
= & \limsup_{k \rightarrow\infty} \int_{\mathbb{Z}^d}\left(|u_k^{\pm}|^{p} \log \left(u_k^{\pm}\right)^{2}\right)^+\,d\mu \\
= & \int_{\mathbb{Z}^d}\left(|u^{\pm}|^{p} \log \left(u^{\pm}\right)^{2}\right)^+\,d\mu,
\end{align*}
and hence
\begin{equation*}
\langle J_{s,2}^{\prime}\left(u\right), u^{\pm}\rangle=\|u^{\pm}\|_{H^{s,2}}^{2}+b\left\|\nabla^s u\right\|_{2}^{2}\left(\|\nabla^s u^{\pm}\|_{2}^{2}-\frac{1}{2}K(u)\right)-\frac{a}{2}K(u)-\int_{\mathbb{Z}^d}|u^{\pm}|^{p} \log \left(u^{\pm}\right)^{2}\,d\mu\leq0.
\end{equation*}
Then by Lemma \ref{l5}, there exist $r, t \in(0,1]$ such that $r u^{+}+t u^{-} \in \mathcal{M}_{s,2}$. We claim that $r=t=1$. Indeed, without loss of generality, we assume that $t<1$. Then by Proposition \ref{o}, we get that
$$
\begin{aligned}
\int_{\mathbb{Z}^d} \left|\nabla^s(ru^{+}+tu^{-})\right|^2 d \mu=&r^2\int_{\mathbb{Z}^d} \left|\nabla^s u^{+}\right|^2 d \mu+t^2\int_{\mathbb{Z}^d}\left|\nabla^s u^{-}\right|^2 d \mu-rtK(u)\\ <&\int_{\mathbb{Z}^d} \left|\nabla^s u^{+}\right|^2 d \mu+\int_{\mathbb{Z}^d}\left|\nabla^s u^{-}\right|^2 d \mu-K(u)\\ =&\int_{\mathbb{Z}^d} \left|\nabla^s u\right|^2 d \mu,
\end{aligned}
$$
which implies that
$$\left\|\nabla^s(r u^{+}+t u^{-})\right\|_{2}^{4}<\left\|\nabla^s u\right\|_{2}^{4},$$
and 
$$\|ru^{+}+t u^{-}\|_{H^{s,2}}^{2}<\|u\|_{H^{s,2}}^{2}.$$
Then we deduce that
$$
\begin{aligned}
m\leq &J_{s,2}(ru^{+}+tu^{-})=J_{s,2}\left(ru^{+}+tu^{-}\right)-\frac{1}{p} \left\langle J_{s,2}^{\prime}\left(ru^{+}+tu^{-}\right), ru^{+}+tu^{-}\right\rangle\\=& \left(\frac{1}{2}-\frac{1}{p}\right)\|ru^{+}+t u^{-}\|_{H^{s,2}}^{2}+b\left(\frac{1}{4}-\frac{1}{p}\right)\left\|\nabla^s(r u^{+}+t u^{-})\right\|_{2}^{4}+\frac{2}{p^2}\int_{\mathbb{Z}^d} |ru^{+}+t u^{-}|^p\,d\mu\\ <&\left(\frac{1}{2}-\frac{1}{p}\right)\|u\|_{H^{s,2}}^{2}+b\left(\frac{1}{4}-\frac{1}{p}\right)\left\|\nabla^s u\right\|_{2}^{4}+\frac{2}{p^2}\int_{\mathbb{Z}^d} |u|^p\,d\mu\\ \leq &
\limsup_{k \rightarrow\infty}\left[\left(\frac{1}{2}-\frac{1}{p}\right)\|u_k\|_{H^{s,2}}^{2}+b\left(\frac{1}{4}-\frac{1}{p}\right)\left\|\nabla^s u_k\right\|_{2}^{4}+\frac{2}{p^2}\int_{\mathbb{Z}^d} |u_k|^p\,d\mu\right]\\ =&
\limsup_{k \rightarrow\infty}\left[J_{s,2}\left(u_{k}\right)-\frac{1}{p} \langle J_{s,2}^{\prime}\left(u_{k}\right), u_{k}\rangle\right] \\=& m.
\end{aligned}
$$
This is a contradiction. Thus $r=t=1$, and we get that $u \in \mathcal{M}_{s,2}$. In the following, we prove that $J_{s,2}\left(u\right)=m>0.$
$$
\begin{aligned}
m\leq &J_{s,2}(u)=J_{s,2}\left(u\right)-\frac{1}{p} \left\langle J_{s,2}^{\prime}\left(u\right), u\right\rangle\\=&\left(\frac{1}{2}-\frac{1}{p}\right)\|u\|_{H^{s,2}}^{2}+b\left(\frac{1}{4}-\frac{1}{p}\right)\left\|\nabla^s u\right\|_{2}^{4}+\frac{2}{p^2}\int_{\mathbb{Z}^d} |u|^p\,d\mu\\ \leq &
\limsup_{k \rightarrow\infty}\left[\left(\frac{1}{2}-\frac{1}{p}\right)\|u_k\|_{H^{s,2}}^{2}+b\left(\frac{1}{4}-\frac{1}{p}\right)\left\|\nabla^s u_k\right\|_{2}^{4}+\frac{2}{p^2}\int_{\mathbb{Z}^d} |u_k|^p\,d\mu\right]\\ =&
\limsup_{k \rightarrow\infty}\left[J_{s,2}\left(u_{k}\right)-\frac{1}{p} \langle J_{s,2}^{\prime}\left(u_{k}\right), u_{k}\rangle\right] \\=& m.
\end{aligned}
$$
Then we have that $J_{s,2}(u)=m>0.$ The proof is completed.
\end{proof}

{\bf Proof of Theorem \ref{t2}.} 
By Lemma \ref{l9}, we only need to prove that $J'_{s,2}(u)=0$. We prove this result by contradiction, suppose $J'_{s,2}(u)\neq 0$. Then there exists $\phi\in H^{s,2}$ such that
$$
\langle J'_{s,2}(u),\phi\rangle\leq -1.
$$
Therefore there exists $0<\varepsilon_1<1$ small enough such that
$$
\left\langle J_{s,2}^{\prime}\left(r u^{+}+t u^{-}+\sigma \phi\right), \phi\right\rangle\leq-\frac{1}{2},\quad |r-1|+|t-1|+|\sigma| \leq \varepsilon_1.
$$
Let $\eta$ be a cutoff function defined as
$$
\eta(r, t)= 
\begin{cases}1, & |r-1| \leq \frac{1}{2} \varepsilon_1 \text { and }|t-1| \leq \frac{1}{2} \varepsilon_1, \\ 0, & |r-1| \geq \varepsilon_1 \text { or }|t-1| \geq \varepsilon_1.
\end{cases}
$$
Now we consider $J_{s,2}\left(r u^{+}+t u^{-}+\varepsilon_1 \eta(r, t) \phi\right)$.
If $|r-1| \geq \varepsilon_1$ or $|t-1| \geq \varepsilon_1,\,  \eta(r, t)=0$, then we have
$$J_{s,2}\left(r u^{+}+t u^{-}+\varepsilon_1 \eta(r, t) \phi\right) =J_{s,2}\left(r u^{+}+t u^{-}\right).$$
If $|r-1| \leq \varepsilon_1$ and $|t-1| \leq \varepsilon_1$, we have
$$
\begin{aligned}
J_{s,2}\left(r u^{+}+t u^{-}+\varepsilon_1 \eta(r, t) \phi\right) &  =J_{s,2}\left(r u^{+}+t u^{-}\right)+\int_{0}^{1} \left\langle J_{s,2}^{\prime}\left(r u^{+}+t u^{-}+\sigma \varepsilon_1 \eta(r, t) \phi\right),\varepsilon_1 \eta(r, t) \phi\right\rangle\, d \sigma \\
& \leq J_{s,2}\left(r u^{+}+t u^{-}\right)-\frac{1}{2} \varepsilon_1 \eta(r, t).
\end{aligned}
$$
Since $u \in \mathcal{M}_{s,2}$, for $(r, t) \neq(1,1)$, by Lemma \ref{l3}, we have 
$$
J_{s,2}\left(r u^{+}+t u^{-}+\varepsilon_1 \eta(r, t) \phi\right) \leq J_{s,2}\left(r u^{+}+t u^{-}\right)<J_{s,2}(u)=m.
$$
For $(r, t)=(1,1)$,
$$
J_{s,2}\left(r u^{+}+t u^{-}+\varepsilon \eta(r, t) \phi\right) \leq J_{s,2}\left(r u^{+}+t u^{-}\right)-\frac{1}{2} \varepsilon_1 \eta(1,1)=J_{s,2}(u)-\frac{1}{2} \varepsilon_1<J_{s,2}(u)=m.
$$
Then for $0<\delta<1-\varepsilon_1$, we have 
\begin{equation}\label{30}
\sup _{\delta \leq r, t \leq 2-\delta} J_{s,2}\left(r u^{+}+t u^{-}+\varepsilon_1 \eta(r, t) \phi\right)<J_{s,2}(u)=m.    
\end{equation}
Let $v=r u^{+}+t u^{-}+\varepsilon_1 \eta(r, t) \phi$. Define
$$
\Phi_{1}(r, t)=\left\langle J_{s,2}^{\prime}(v), v^{+}\right\rangle,\qquad \Phi_{2}(r, t)= \left\langle J_{s,2}^{\prime}(v), v^{-}\right\rangle.
$$
Since $u\in\mathcal{M}_{s,2}$, we have
\begin{equation}\label{3.7}
\begin{aligned}
\int_{\mathbb{Z}^d}\left|u^{+}\right|^{p} \log \left(u^{+}\right)^{2}\,d\mu= &\|u^{+}\|_{H^{s,2}}^{2}+b\left\|\nabla^s u^{+}\right\|_{2}^{4}-\frac{a}{2}K(u)+\frac{b}{2}K^2(u)
+b\left\|\nabla^s u^{+}\right\|_{2}^{2}\left\|\nabla^s u^{-}\right\|_{2}^{2}\\&-\frac{b}{2}K(u)\left(\left\|\nabla^s u^{+}\right\|_{2}^{2}+\left\|\nabla^s u^{-}\right\|_{2}^{2}\right)
-bK(u)\|\nabla^s u^{+}\|^2_2,
\end{aligned}
\end{equation}
and 
\begin{equation}\label{4.7}
\begin{aligned}
\int_{\mathbb{Z}^d}\left|u^{-}\right|^{p} \log \left(u^{-}\right)^{2}\,d\mu= &\|u^{-}\|_{H^{s,2}}^{2}+b\left\|\nabla^s u^{-}\right\|_{2}^{4}-\frac{a}{2}K(u)+\frac{b}{2}K^2(u)
+b\left\|\nabla^s u^{+}\right\|_{2}^{2}\left\|\nabla^s u^{-}\right\|_{2}^{2}\\&-\frac{b}{2}K(u)\left(\left\|\nabla^s u^{+}\right\|_{2}^{2}+\left\|\nabla^s u^{-}\right\|_{2}^{2}\right)
-bK(u)\|\nabla^s u^{-}\|^2_2.
\end{aligned}
\end{equation}
By the definition of $\eta$, for $r=\delta<1-\varepsilon_1$ and $t \in[\delta, 2-\delta]$, we have $\eta(r, t)=0$ and $r\leq t$. Hence by (\ref{3.7}) and $p>6$, we get that
\begin{align*}
&\Phi_{1}(\delta, t)=\langle J_{s,2}^{\prime}(\delta u^++tu^-),\delta u^{+}\rangle\\=&\delta^{2}\|u^{+}\|_{H^{s,2}}^{2}+b \delta^{4}\left\|\nabla^s u^{+}\right\|_{2}^{4}-\delta^p\int_{\mathbb{Z}^d}\left|u^{+}\right|^{p} \log \left(u^{+}\right)^{2}\,d\mu-\delta^p\log \delta^2\int_{\mathbb{Z}^d}\left|u^{+}\right|^{p}\,d\mu-\frac{a}{2}\delta tK(u)\\&+\frac{b}{2}\delta^2t^2K^2(u)
+b \delta^{2} t^{2}\left\|\nabla^s u^{+}\right\|_{2}^{2}\left\|\nabla^s u^{-}\right\|_{2}^{2}-\frac{b}{2}\delta tK(u)\left(\delta^2\left\|\nabla^s u^{+}\right\|_{2}^{2}+t^2\left\|\nabla^s u^{-}\right\|_{2}^{2}\right)
\\&-b\delta^3tK(u)\|\nabla^s u^{+}\|^2_2\\ \geq &\delta^{2}\|u^{+}\|_{H^{s,2}}^{2}+b \delta^{4}\left\|\nabla^s u^{+}\right\|_{2}^{4}-\delta^p\int_{\mathbb{Z}^d}\left|u^{+}\right|^{p} \log \left(u^{+}\right)^{2}\,d\mu-\delta^p\log \delta^2\int_{\mathbb{Z}^d}\left|u^{+}\right|^{p}\,d\mu-\frac{a}{2}\delta^2K(u)\\&+\frac{b}{2}\delta^4K^2(u)
+b \delta^{4}\left\|\nabla^s u^{+}\right\|_{2}^{2}\left\|\nabla^s u^{-}\right\|_{2}^{2}-\frac{b}{2}\delta^4K(u)\left(\left\|\nabla^s u^{+}\right\|_{2}^{2}+\left\|\nabla^s u^{-}\right\|_{2}^{2}\right)
-b\delta^4K(u)\|\nabla^s u^{+}\|^2_2\\=&(\delta^2-\delta^p)\|u^{+}\|_{H^{s,2}}^{2}+b(\delta^{4}-\delta^p)\left\|\nabla^s u^{+}\right\|_{2}^{4}-\delta^p\log \delta^2\int_{\mathbb{Z}^d}\left|u^{+}\right|^{p}\,d\mu-\frac{a}{2}(\delta^2-\delta^p)K(u)\\&+\frac{b}{2}(\delta^4-\delta^p)K^2(u)
+b (\delta^{4}-\delta^p)\left\|\nabla^s u^{+}\right\|_{2}^{2}\left\|\nabla^s u^{-}\right\|_{2}^{2}-\frac{b}{2}(\delta^4-\delta^p)K(u)\left(\left\|\nabla^s u^{+}\right\|_{2}^{2}+\left\|\nabla^s u^{-}\right\|_{2}^{2}\right)
\\&-b(\delta^4-\delta^p)K(u)\|\nabla^s u^{+}\|^2_2\\>&-\delta^p\log \delta^2\int_{\mathbb{Z}^d}\left|u^{+}\right|^{p}\,d\mu\\>&0.
\end{align*}
For $r=2-\delta>1+\varepsilon_1$ and $t \in[\delta, 2-\delta]$, we have $\eta(r, t)=0$ and $r\geq t$. Similarly, we have 
\begin{align*}
&\Phi_{1}(2-\delta, t)=\langle J_{s,2}^{\prime}((2-\delta)u^++tu^-),(2-\delta)u^{+}\rangle\\=&(2-\delta)^{2}\|u^{+}\|_{H^{s,2}}^{2}+b (2-\delta)^{4}\left\|\nabla^s u^{+}\right\|_{2}^{4}-(2-\delta)^p\int_{\mathbb{Z}^d}\left|u^{+}\right|^{p} \log \left(u^{+}\right)^{2}\,d\mu\\&-(2-\delta)^p\log (2-\delta)^2\int_{\mathbb{Z}^d}\left|u^{+}\right|^{p}\,d\mu-\frac{a}{2}(2-\delta)tK(u)+\frac{b}{2}(2-\delta)^2t^2K^2(u)
\\&+b (2-\delta)^{2} t^{2}\left\|\nabla^s u^{+}\right\|_{2}^{2}\left\|\nabla^s u^{-}\right\|_{2}^{2}-\frac{b}{2}(2-\delta)tK(u)\left((2-\delta)^2\left\|\nabla^s u^{+}\right\|_{2}^{2}+t^2\left\|\nabla^s u^{-}\right\|_{2}^{2}\right)
\\&-b(2-\delta)^3tK(u)\|\nabla^s u^{+}\|^2_2\\ \leq &(2-\delta)^{2}\|u^{+}\|_{H^{s,2}}^{2}+b (2-\delta)^{4}\left\|\nabla^s u^{+}\right\|_{2}^{4}-(2-\delta)^p\int_{\mathbb{Z}^d}\left|u^{+}\right|^{p} \log \left(u^{+}\right)^{2}\,d\mu\\&-(2-\delta)^p\log (2-\delta)^2\int_{\mathbb{Z}^d}\left|u^{+}\right|^{p}\,d\mu-\frac{a}{2}(2-\delta)^2K(u)+\frac{b}{2}(2-\delta)^4K^2(u)
\\&+b (2-\delta)^{4}\left\|\nabla^s u^{+}\right\|_{2}^{2}\left\|\nabla^s u^{-}\right\|_{2}^{2}-\frac{b}{2}(2-\delta)^4K(u)\left(\left\|\nabla^s u^{+}\right\|_{2}^{2}+\left\|\nabla^s u^{-}\right\|_{2}^{2}\right)
-b(2-\delta)^4K(u)\|\nabla^s u^{+}\|^2_2\\=&\left[(2-\delta)^2-(2-\delta)^p\right]\|u^{+}\|_{H^{s,2}}^{2}+b\left[(2-\delta)^{4}-(2-\delta)^p\right]\left\|\nabla^s u^{+}\right\|_{2}^{4}-(2-\delta)^p\log (2-\delta)^2\int_{\mathbb{Z}^d}\left|u^{+}\right|^{p}\,d\mu\\&-\frac{a}{2}\left[(2-\delta)^2-(2-\delta)^p\right]K(u)+\frac{b}{2}\left[(2-\delta)^2-(2-\delta)^p\right]K^2(u)
\\&+b\left[(2-\delta)^{4}-(2-\delta)^p\right]\left\|\nabla^s u^{+}\right\|_{2}^{2}\left\|\nabla^s u^{-}\right\|_{2}^{2}-\frac{b}{2}\left[(2-\delta)^4-(2-\delta)^p\right]K(u)\left(\left\|\nabla^s u^{+}\right\|_{2}^{2}+\left\|\nabla^s u^{-}\right\|_{2}^{2}\right)
\\&-b\left[(2-\delta)^4-(2-\delta)^p\right]K(u)\|\nabla^s u^{+}\|^2_2\\<&-(2-\delta)^p\log (2-\delta)^2\int_{\mathbb{Z}^d}\left|u^{+}\right|^{p}\,d\mu\\<& 0.
\end{align*}
Hence, we have
$$
\Phi_{1}(\delta, t)>0,\qquad \Phi_{1}(2-\delta, t)<0,\quad t \in[\delta, 2-\delta].
$$
Similarly, based on the arguments as above and (\ref{4.7}), we obtain that
$$
\Phi_{2}(r, \delta)>0,\qquad \Phi_{2}(r, 2-\delta)<0,\quad r \in[\delta, 2-\delta].
$$
By the Miranda's theorem \cite{K}, there exists $\left(r_{0}, t_{0}\right) \in(\delta, 2-\delta) \times(\delta, 2-\delta)$ such that $\Phi_{1}(r_0,t_0)=0$ and $\Phi_2(r_0,t_0)=0$, which means that $$\widetilde{u}=r_{0} u^{+}+t_{0} u^{-}+\varepsilon_1 \eta\left(r_{0}, t_{0}\right) \phi \in\mathcal{M}_{s,2}.$$ However, by (\ref{30}), we get that $m\leq J_{s,2}(\widetilde{u})<m,$ which is a contradiction. Therefore, we deduce that $J'_{s,2}(u)=0$. The proof is completed.
\qed

\
\

Finally, we prove the multiplicity of solutions to the equation (\ref{0.2}).

\
\

{\bf Proof of Theorem \ref{t3}.}
Let $u$ and $v$ be the ground state solution and the ground state sign-changing solution to the equation (\ref{0.2}), respectively. We first show that $J_{s,2}(v)>2J_{s,2}(u)$. In fact, since $v \in \mathcal{M}_{s,2}$ satisfies $J_{s,2}(v)=m$, we have $v^{ \pm} \neq 0$. Then by Lemma \ref{lz}, there exists a unique $r_{v}>0$ such that $r_{v} v^{+} \in \mathcal{N}_{s,2}$, and a unique $t_{v}>0$ such that $t_{v} v^{-} \in \mathcal{N}_{s,2}$. Moreover, by Theorem \ref{t1},  we have $u \in \mathcal{N}_{s,2}$ satisfies $J_{s,2}(u)=\hat{c}$.

By the fact $v^{\pm}\neq 0$, we have that $K(v)<0$. A direct calculation yields that
$$
\begin{aligned}
&J_{s,2}\left(r_{v}v^++t_{v}v^-\right)\\=&J_{s,2}\left(r_{v}v^{+}\right)+J_{s,2}\left(t_{v}v^{-}\right)-\frac{a}{2}r_{v}t_{v}K(v)+\frac{b}{4}r_{v}^2t_{v}^2K^2(v)+\frac{b}{2}\left\|\nabla^s (r_{v}v^{+})\right\|_{2}^{2}\left\|\nabla^s (t_{v}v^{-})\right\|_{2}^{2}\\ &-\frac{b}{2}r_{v}t_{v}K(v)\left(\left\|\nabla^s (r_{v}v^{+})\right\|_{2}^{2}+\left\|\nabla^s (t_{v}v^{-})\right\|_{2}^{2}\right)\\>&J_{s,2}\left(r_{v}v^{+}\right)+J_{s,2}\left(t_{v}v^{-}\right).
\end{aligned}
$$
By Lemma \ref{l3}, we get that
$$
m=J_{s,2}\left(v\right) \geq J_{s,2}\left(r_{v} v^{+}+t_{v} v^{-}\right)>J_{s,2}\left(r_{v} v^{+}\right)+J_{s,2}\left(t_{v} v^{-}\right) \geq 2J_{s,2}(u)=2\hat{c}.
$$

In the following, we prove the multiplicity of solutions to the equation (\ref{0.2}).
Let $w\in H^{s,2}$ be a nontrivial weak solution to the equation \ref{0.2}. Then by Lemma \ref{li}, $w$ is a pointwise solution to the equation (\ref{0.2}). It is clear that $z=(-w)\in H^{s,2}\backslash\{0\}$. Moreover, we have

$$\begin{aligned}
    (-\Delta)^{s} z(x)=& \sum_{y \in \mathbb{Z}^d, y \neq x} w_{s}(x, y)(z(x)-z(y))\\ =&-\sum_{y \in \mathbb{Z}^d, y \neq x} w_{s}(x, y)(w(x)-w(y))\\=&-(-\Delta)^{s} w(x).
\end{aligned}$$
Then we get that
$$\begin{aligned}
    \left(a+b \int_{\mathbb{Z}^d}|\nabla^s z|^{2} d \mu\right) (-\Delta)^s z=&-\left(a+b \int_{\mathbb{Z}^d}|\nabla^s w|^{2} d \mu\right) (-\Delta)^s w\\=&h(x) w-|w|^{p-2}w \log w^{2}\\=&h(x)w+|-w|^{p-2}(-w) \log (-w)^{2}\\ =&-h(x)z+|z|^{p-2} z \log v^{2}.
\end{aligned}$$
This means that $z$ is also a nontrivial weak solution to the equation (\ref{0.2}), which differs from $w$. Since $u$ is a ground state solution and $v$ is a ground state sign-changing solution to the equation (\ref{0.2}), we derive that $-u$ is a nontrivial weak solution and $-v$ is a sign-changing solution to the equation (\ref{0.2}) respectively. Clearly, $-u \not \equiv u$ and $-v \not\equiv v$.

Since $m>2 \hat{c}> \hat{c}$, we deduce that $\pm u$ and $\pm v$ are four different nontrivial weak solutions to the equation (\ref{0.2}). \qed

\section{Acknowledgements}
The author thanks the referee for helpful comments and valuable suggestions on this paper. Moreover, the author is supported by the National Natural Science Foundation of China, no.12401135.

\
\

{ \bf Declaration}
The author declares that there are no conflict of interest regarding the publication of this paper.

\
\

{ \bf Data availability}
No data was used for the research described in this paper.

\end{document}